\documentclass[a4paper, english, 11pt, twoside]{amsart}

\usepackage{babel}
\usepackage[T1]{fontenc}
\usepackage{amssymb}
\usepackage{amsmath, amsfonts, mathrsfs, amscd}
\usepackage{xcolor}
\usepackage{eucal}
\usepackage[all]{xy}
\usepackage{calligra}
\usepackage[all]{xypic}
\usepackage{graphicx}
\usepackage{hyperref}
\usepackage{amsthm}

\newtheorem{theorem-intro}{Theorem}
\newtheorem{theorem}{Theorem}[section]
\newtheorem{lemma}[theorem]{Lemma}
\newtheorem{proposition}[theorem]{Proposition}
\newtheorem{corollary}[theorem]{Corollary}

\newtheorem*{questions}{Questions}

\theoremstyle{definition}
\newtheorem{definition}[theorem]{Definition}
\newtheorem{remark}[theorem]{Remark}

\newtheorem*{example}{Example}

\def\pf{\begin{proof}}
\def\epf{\end{proof}}

\newcommand{\Na}{\mathbb{N}}
\newcommand{\Z}{\mathbb{Z}}
\newcommand{\Q}{\mathbb{Q}}
\newcommand{\Co}{\mathbb{C}}

\newcommand{\p}{\mathfrak{p}}
\newcommand{\q}{\mathfrak{q}}
\newcommand{\Ow}{\mathcal{O}}

\newcommand\Ker{\operatorname{Ker}}
\newcommand\Ima{\operatorname{Im}}

\newcommand\Rep{\operatorname{Rep}}

\newcommand{\op}{\frac{1}{\pi}}
\newcommand{\ot}{\otimes}

\newcommand{\Oint}{\mathcal{O}}
\newcommand{\Bas}{\mathcal{B}}

\definecolor{azul}{RGB}{0, 47, 103}

\allowdisplaybreaks

\definecolor{rojo}{rgb}{1,0,0}

\DeclareMathAlphabet{\mathpzc}{OT1}{pzc}{m}{it}

\begin{document}

\title[Orders of Nikshych's Hopf algebra]{Orders of Nikshych's Hopf algebra}

\author[J. Cuadra and E. Meir]{Juan Cuadra and Ehud Meir}

\address{J. Cuadra: Universidad de Almer\'\i a, Dpto. Matem\'aticas. E04120 Almer\'\i a, Spain}
\email{jcdiaz@ual.es}

\address{E. Meir: Department of Mathematics, University of Hamburg, Bundesstrasse 55, 20146 Hamburg, Germany}
\email{meirehud@gmail.com}

\begin{abstract}
Let $p$ be an odd prime number and $K$ a number field having a primitive $p$-th root of unity $\zeta_p$.
We prove that Nikshych's non group-theoretical Hopf algebra $H_p$, which is defined over $\Q(\zeta_p)$, admits a Hopf order over the ring of integers $\Oint_K$ if and only if there is an ideal $I$ of $\Oint_K$ such that $I^{2(p-1)} = (p)$.
This condition does not hold in a cyclotomic field. Hence this gives an example of a semisimple Hopf algebra over a number field not admitting a Hopf order over any cyclotomic ring of integers.
Moreover, we show that, when a Hopf order over $\Oint_K$ exists, it is unique and we describe it explicitly.
\end{abstract}

\maketitle

\section*{Introduction}
\setcounter{equation}{0}

Many results in the Representation Theory of Finite Groups exploit the fact that the complex group algebra $\Co G$ of a finite group $G$ is defined over the integers or, more generally, over the ring of integers $\Oint_K$ of a number field $K$. In other terms, $\Oint_K G$ is an algebra order of $\Co G$; indeed a Hopf (algebra) order. A prominent role is played by cyclotomic fields: for example, the celebrated Brauer's splitting field theorem states that any irreducible representation of $KG$ can be realized in $K(\omega)$, with $\omega$ a primitive root of unity of order equals $\textrm{exp}\, G$ (see \cite[Theorem 15.16, Corollary 15.18]{CR}). \par \smallskip

Kaplansky's sixth conjecture, still unsolved, is a generalization of Frobenius Theorem for groups. It asserts that in a complex semisimple Hopf algebra $H$ the dimension of every irreducible representation of $H$ divides the dimension of $H$. Larson gave a positive answer in \cite{L2} if $H$ admits a Hopf order over a number ring. Motivated by this result, in \cite{CM} we addressed the question as to whether any complex semisimple Hopf algebra admits a Hopf order over a number ring. In the dimensions less than $36$ in which the classification is complete ($24$ and $32$ are still open) it turns out that all semisimple Hopf algebras are defined over cyclotomic rings of integers, see \cite[Subsection 2.4]{CM} for an account. However, we exhibited in \cite{CM} an example in dimension $36$ that does not admit a Hopf order over any number ring, although it satisfies the conjecture. \par \smallskip

As a continuation of our previous work we investigate in this paper the problem of definability of semisimple Hopf algebras over cyclotomic ring of integers. Let $H$ be a semisimple Hopf algebra over a number field $K$ and suppose that $H$ has a Hopf order over some number ring.
Does $H$ admit a Hopf order over a cyclotomic ring of integers contained in $K$? Our main result gives a negative answer for the family of non group-theoretical semisimple Hopf algebras $\{H_p\}$, with $p$ an odd prime, constructed by Nikshych (see \cite{N1}). The dimension of $H_p$ is $4p^2$ (so in particular the dimension of $H_3$ is 36). These Hopf algebras were not constructed explicitly but through a tensor category and a fiber functor. The representation category $\Rep(H_p)$ was obtained by equivariantization by $C_2$ from $\Rep(A_p)$, with $A_p$ the Hopf algebra studied by Masuoka in \cite{M1}. Using Tannaka reconstruction, in Section \ref{Nikshych} we describe $H_p$ completely as follows:

\begin{theorem-intro}
Let $\zeta_p \in \Co$ be a primitive $p$-th root of unity. The Hopf algebra $H_p$ is generated, as an algebra over $\Co$, by the elements $e_0,e_1,u_a,u_b,v_a,v_b$ and $g$ subject to the following relations:
$$\begin{array}{lllll}
e_0+e_1=1,       & \hspace{3mm} e_0e_1=e_1e_0=0,       &                  &                             & \vspace{3pt}  \\
u_a^p=u_b^p=e_0, & \hspace{3mm} e_0u_a=u_a,         &  e_0u_b=u_b,  & \hspace{3mm} u_au_b=u_bu_a, & \vspace{3pt} \\
v_a^p=v_b^p=e_1, & \hspace{3mm} e_1v_a=v_a, &  e_1v_b=v_b,  & \hspace{3mm} v_av_b = \zeta_p v_bv_a, & \vspace{3pt} \\
g^2=1,           & \hspace{3mm} gu_a= u_bg,            &  gu_b= u_ag,     & \hspace{3mm} gv_a = v_ag,   & \hspace{3mm} gv_b = v_bg. \\
\end{array}$$
The comultiplication, counit, and antipode of $H_p$ is given by the following formulas:
$$\begin{array}{lll}
\Delta(u_a) = u_a\otimes u_a + v_a\otimes v_a,       & \hspace{5mm} \varepsilon(u_a)=1, & \hspace{5mm} S(u_a) = u_a^{p-1}, \vspace{3pt} \\
\Delta(u_b) = u_b\otimes u_b + v_b\otimes v_b^{p-1}, & \hspace{5mm} \varepsilon(u_b)=1, & \hspace{5mm} S(u_b)=u_b^{p-1}, \vspace{3pt} \\
\Delta(v_a) = u_a \otimes v_a + v_a\otimes u_a,      & \hspace{5mm} \varepsilon(v_a)=0, & \hspace{5mm} S(v_a) = v_a^{p-1}, \vspace{3pt} \\
\Delta(v_b) = u_b\otimes v_b + v_b\otimes u_b^{p-1}, & \hspace{5mm} \varepsilon(v_b)=0, & \hspace{5mm} S(v_b)=v_b.
\end{array}$$
The comultiplication of $g$ is given by
$$\begin{array}{ll}
\Delta(g)  = & {\displaystyle \frac{1}{p^2}\sum_{i,j,k,l} \zeta^{kj-il}_p gu_a^i u_b^j \otimes gu_a^k u_b^l + \frac{1}{p}\sum_{k,l} \zeta^{-(k+l)k}_p gu_a^k u_b^l \otimes gv_a^{k+l} v_b^{k+l}} \vspace{3pt} \\
 & {\displaystyle + \frac{1}{p} \sum_{k,l} \zeta^{k(k+l)}_p gv_a^{k+l}v_b^{(p-1)(k+l)} \otimes gu_a^k u_b^l + \frac{1}{p} \sum_{k,l} gv_a^k v_b^l \otimes gv_a^{(p-1)l}v_b^k.}
\end{array}$$
The counit and antipode of $g$ are $\varepsilon(g)=1$ and $S(g)=g$.
\end{theorem-intro}

In Section 4 we delve into the structure of $H_p$: we describe its irreducible (co)re\-pre\-sen\-ta\-tions and attached (co)characters, its Hopf automorphisms, and we show that $H_p$ is self-dual. \par \smallskip

The set
$$\Bas:=\{u_a^iu_b^j\}\cup\{v_a^iv_b^j\}\cup\{gu_a^iu_b^j\}\cup\{gv_a^iv_b^j\}$$
is a basis of $H_p$. All structure constants of $H_p$ in this basis belong to $\Q(\zeta_p)$. Hence $H_p$ is defined over $\Q(\zeta_p)$. Our main result states:

\begin{theorem-intro}\label{thintro2}
Let $K$ be a number field containing a primitive $p$-th root of unity $\zeta_p$. Consider $H_p$ as defined over $K$. Then, $H_p$ admits a Hopf order over $\Oint_K$, which must be unique, if and only if there is an ideal $I$ of $\Oint_K$ such that $I^{2(p-1)} = (p)$. In particular, $K$ can not be a cyclotomic field (nor an abelian extension of $\Q$) if a Hopf order exists.
\end{theorem-intro}

This theorem implies that Nikshych's Hopf algebras behave rather differently than group algebras.
Firstly, all group algebras are already defined over $\Z$. Secondly, the number of Hopf orders of a group algebra over $\Oint_K$
depends on $K$, and in some cases it is not bounded (see for example the classification of orders of the group algebras of the cyclic groups of prime orders in Section \ref{ordercp}). \par \smallskip

The main result is contained in Section 5. We outline the strategy to prove it and construct the Hopf order.
The element $h:=u_a+v_a$ is a group-like element of $H_p$ and generate a Hopf subalgebra isomorphic to $KC_p$.
If $X$ is a Hopf order of $H_p$ over $\Oint_K$, then $X \cap KC_p$ is a Hopf order of $KC_p$.
The Hopf orders of the latter are known by the works of Greither, Larson, Tate and Oort (we review their description in Section \ref{ordercp}, after the preliminaries). They are given by ideals $I$ of $\Oint_K$ containing $\zeta_p-1$, see Formula \ref{LarHI}. Denoting by $H(I)$ the corresponding Hopf order, the $\Oint_K$-submodule of integrals of $H(I)$ is $\frac{1}{p}I^{p-1}\sum_i h^i$. This determines uniquely the Hopf orders of $KC_p$.
On the other hand, any Hopf order must contain certain elements arising from characters and cocharacters.
The proof of the main result is based on the interaction between the order $X$ of $H_p$ and the order $X\cap KC_p$ of $KC_p$. We exhibit certain elements which must be in $X$. We then conclude that necessarily $\frac{1}{\sqrt{p}} \sum_{i} h^i \in X \cap KC_p$,
and by the classification of orders mentioned above, we conclude that some more elements must lie in $X\cap KC_p$ and therefore in $X$.
We then show that these elements generate an order of $H_p$, which thus must be a minimal order.
We then use the self-duality of $H_p$ and conclude that there is also a maximal order.
A result of Larson (see Proposition \ref{Larson2}) now implies that the two orders must be equal, and therefore we only have one order.
The necessity of the existence of an ideal $I$ of $\Oint_K$ such that  $I^{2(p-1)}=(p)$ arises from the following consideration:
We prove that the set of integrals of $X\cap KC_p$ is exactly $\Oint_K\big(\frac{1}{\sqrt{p}} \sum_{i} h^i\big)$.
We write $J=\{x\in K\, |\, x(h-1)\in X\}$. By the classification in Section \ref{ordercp} we find out that $I:=J^{-1}$ must satisfy $I^{2(p-1)}=(p)$. The unique Hopf order of $H_p$ is the $\Oint_K$-subalgebra of $H_p$ generated by $e_0,e_1,g,J(u_a-e_0), J(u_b-e_0), J(v_a-e_1),$ and $J(v_b-e_1)$. \par \smallskip

In Section 6 we study the problem of definability over cyclotomic ring of integers of $H_p$ but now considered as a complex Hopf algebra. Since $H_p$ is already defined over a number field $K$, the question reads now as follows. Let $L/K$ be a Galois extension. Could a $L/K$-form of $H_p$ admit an order over some cyclotomic ring of integers? Namely, could there be another Hopf algebra $H'_p$ over $K$ such that $H'_p \otimes_K L \simeq H_p \otimes_K L$ and $H'_p$ admits an order over some cyclotomic ring of integers? The following result gives a number theoretical condition under which the answer is affirmative:

\begin{theorem-intro}\label{thintro3}
Let $\zeta_n \in \Co$ be a primitive $n$-th root of unity, with $n$ divisible by $p$. Consider $H_p$ as defined over $\Q(\zeta_n)$. Let $w \in \Z[\zeta_n]$ and $t \in \Co$ be such that $w$ is invertible and $t^2=w({\zeta_p}-1)$. Assume that\vspace{-1pt} there is $d \in \Z[\zeta_n]$ such that $\frac{1}{2}(d+t)\in \Ow_{\Q(\zeta_n,t)}$. Then, $H_p$ admits a $\Q(\zeta_n,t)/\Q(\zeta_n)$-form $H'_p$ which in turn admits an order over $\Z[\zeta_n]$.
\end{theorem-intro}

For $p=7$ and $n=28$ we construct elements $w,t$ and $d$ satisfying this condition. So, $H_7$, as a complex Hopf algebra, admits an order over the cyclotomic ring of integers $\Z[\zeta_{28}]$. \par \smallskip

The following questions on the definability over cyclotomic ring of integers of complex semisimple Hopf algebras remain open:

\begin{questions}
Does there exist a value of $p$ for which Nikshych's Hopf algebra $H_p$, as defined over the complex numbers,  does not admit an order over any cyclotomic ring of integers? More generally, does there exist a complex semisimple Hopf algebra which admits an order over a number ring but not over any cyclotomic ring of integers?
\end{questions}

\section{Preliminaries}\label{prelim}
\setcounter{equation}{0}

Throughout $H$ is a finite-dimensional Hopf algebra over a ground field $K$.
Unless otherwise stated, vector spaces, linear maps, and unadorned tensor products are over $K$.
The identity element of $H$ is denoted by $1_H$ and the comultiplication, counit, and antipode by $\Delta, \varepsilon,$ and $S$ respectively.
Our main references for Hopf algebra theory are \cite{Mo} and \cite{Ra}. \par \smallskip

We next collect from \cite[Subsection 1.2]{CM} several notions and results on Hopf orders that we will need later. We refer the reader to there for the proofs. \par \smallskip

\subsection{Hopf orders} Let $R \subset K$ be a subring and $V$ a finite-dimensional $K$-vector space.
Recall that an order of $V$ over $R$ is a finitely generated and projective $R$-submodule $X$ of $V$ such that the natural map $X \otimes_R K \rightarrow V$ is an isomorphism.
We view $X$ inside $V$ as the image of $X \otimes_R R$. A \textit{Hopf order of $H$ over $R$} is an order $X$ of $H$ such that $1_H \in X$, $XX \subseteq X$,
$\Delta(X)\subseteq X\otimes_{R} X$, $\varepsilon(X) \subseteq R$ and $S(X)\subseteq X$. (Note that $X\otimes_{R} X$ can be identified naturally as an $R$-submodule of $H\otimes H$.)
Equivalently, a Hopf order of $H$ over $R$ is a Hopf algebra $X$ over $R$, which is finitely generated and projective as an $R$-module,
such that $X\otimes_{R} K \simeq H$ as Hopf algebras over $K$. We will assume throughout this subsection that $K$ is a number field and $R=\Oint_K$. A Hopf order without indication of the ground ring means a Hopf order over $R$. \par \smallskip

\begin{proposition}\label{subsquo} Let $X$ be a Hopf order of $H$.
\begin{enumerate}
\item[(i)] The dual order $X^{\star}:=\{\varphi \in H^* : \varphi(X) \subseteq R\}$ is a Hopf order of $H^*$. \vspace{2pt}
\item[(ii)] The natural isomorphism $H \simeq H^{**}$ induces an isomorphism of Hopf orders $X \simeq X^{\star \star}$.
\item[(iii)] If $A$ is a Hopf subalgebra of $H$, then $X\cap A$ is a Hopf order of $A$. \vspace{2pt}
\item[(iv)] If $f:H \rightarrow B$ is a surjective Hopf algebra map, then $f(X)$ is a Hopf order of $B$.
\end{enumerate}
\end{proposition}

An important fact in our study of Hopf orders is that they contain certain elements arising from the characters and cocharacters of the Hopf algebra.

\begin{proposition}\label{character}
Let $X$ be a Hopf order of $H$. Any character of $H$ belongs to $X^{\star}$. As a consequence, any character of $H^*$ belongs to $X$.
\end{proposition}

We will also need the following two results by Larson:

\begin{proposition}\cite[Proposition 2.2]{L2}\label{Larson1}
Let $H$ be a semisimple Hopf algebra over $K$ and $X$ a Hopf order of $H$.
Denote by $\Lambda_X$ and $\Lambda_{X^{\star}}$ the $R$-submodule of left integrals of $X$ and $X^{\star}$ respectively.
Then $\varepsilon(\Lambda_X)\varepsilon(\Lambda_{X^{\star}})=(\dim H)$ as ideals in $R$.
\end{proposition}

\begin{proposition}\cite[Corollary 3.2]{L2}\label{Larson2}
With hypotheses as before, assume that $X$ and $Y$ are Hopf orders of $H$ such that $X\subseteq Y$.
If $\varepsilon(\Lambda_X)=\varepsilon(\Lambda_Y)$, then $X=Y$.
\end{proposition}

\section{Classification of Hopf orders of $KC_p$}\label{ordercp}
\setcounter{equation}{0}

Let $p$ be a prime number and $\zeta$ a primitive $p$-th root of unity. Let $K$ be a number field containing $\zeta$ and $R:=\Oint_K$.
Let $\sigma$ denote a generator of the cyclic group $C_p$. We will describe here all Hopf orders of $KC_p$.
Tate and Oort classified all group schemes of order $p$ over $R$ in \cite[Theorem 3]{TO}.
Their result is more general than classifying Hopf orders over $R$.
However, we will combine it with Greither's result \cite[Lemma 1.2, page 40]{G} to give a more explicit description of all Hopf orders of $KC_p$. \par \smallskip

We begin with the following observation:

\begin{lemma}
Let $X$ be a Hopf order of $KC_p$. Consider the fractional ideal $$J=\{\alpha \in K :\  \alpha(\sigma-1)\in X\}.$$
Then $R \subseteq J \subseteq R\frac{1}{\zeta-1}$.
\end{lemma}

\pf
By Proposition \ref{character}, $\psi(X)\subseteq R$ for any character $\psi$ of $C_p$.
Using the character mapping $\sigma$ to $\zeta$ we see that $J(\zeta-1)\subseteq R$.
Hence $J \subseteq R\frac{1}{\zeta-1}$. \vspace{-1pt} For the other inclusion, notice that $\sigma$ is a character of $(KC_p)^*$.
Then $\sigma \in X$ again by Proposition \ref{character}, and $R(\sigma-1) \subseteq X$.
\epf

The above lemma leads us to the following definition:

\begin{definition}
Let $I$ be an ideal of $R$ containing $\zeta-1$. The {\it global Larson order associated to $I$} is the $R$-submodule of $KC_p$
\begin{equation}\label{LarHI}
H(I)=\bigoplus_{i=0}^{p-1} I^i(\zeta-1)^{-i}(\sigma-1)^i.
\end{equation}
\end{definition}

The name global Larson order will make sense in a few paragraphs. Notice that if $(\zeta-1) \subseteq I\subseteq I'$, then $H(I)\subseteq H(I')$. Even though the Larson orders are orders of the cyclic group algebra, they will play a decisive role in the classification of orders of Nikshych's Hopf algebra in Section \ref{Nik}.

\begin{lemma}\label{glo}
The global Larson orders are Hopf orders of $KC_p$. The set of integrals in $H(I)$ is $\frac{1}{p}I^{p-1}\sum_i \sigma^i$.
\end{lemma}

\pf
We first show that $H(I)$ is closed under multiplication. For this, it is enough to prove that $I^p(\zeta-1)^{-p}(\sigma-1)^p\subseteq H(I)$.
This follows from the fact that the element $x:=\frac{1}{\zeta-1}(\sigma-1)$ satisfies a monic polynomial over $R$ of degree $p.$ We have:
$$1=\sigma^p=\big((\zeta-1)x+1\big)^p=\sum_{k=0}^p \binom{p}{k}(\zeta-1)^kx^k \Longrightarrow \sum_{k=1}^p \binom{p}{k}(\zeta-1)^{k-1}x^k=0.$$
The coefficient of $x^p$ is $(\zeta-1)^{p-1}$ and this equals $p\xi$ for some $\xi \in R$ invertible. Multiplying by $p^{-1}\xi^{-1}$ we obtain the desired polynomial.
On the other hand, it is clear that $1 \in H(I), \varepsilon(H(I))\subseteq R,$ and $S(H(I)) \subseteq H(I)$. It remains to prove
that $\Delta(H(I)) \subset H(I)\otimes_R H(I)$. Since $\Delta$ is an algebra map and $H(I)$ is closed under multiplication,
it suffices to check that $\Delta(rx) \in H(I)\otimes_R H(I)$ for every $r\in I$. A direct calculation reveals that
$$\Delta(rx) = rx\otimes 1 + 1\otimes rx + (\zeta-1)x\otimes rx.$$
The first two summands clearly belong to $H(I)\otimes_R H(I)$ and the third summand too because $\zeta-1 \in I$. \par \smallskip

To prove the statement about the integrals, notice that the integral $\frac{1}{p}\sum_i \sigma^i$ equals an invertible element times a monic polynomial $f$ of degree $p-1$ in $x$.
This can be seen by the following calculation:
\begin{equation}\label{integral}
\frac{1}{p}\sum_i \sigma^i = \frac{1}{p}\frac{\big((\zeta-1)x + 1\big)^p -1}{(\zeta-1)x}=\sum_{k=1}^p \frac{1}{p}\binom{p}{k}(\zeta-1)^{k-1}x^{k-1}.
\end{equation}
The fractional expression is just symbolic as $(\zeta-1)x$ is not necessarily invertible. The powers of $x$ in the right-hand side term have coefficients in $R$.
Observe that $p$ divides $\binom{p}{k}$ for $k=1,\ldots,p-1$. For $k=p$ the coefficient of $x^{p-1}$ is $(\zeta-1)^{p-1}=p\xi$ with $\xi \in R$ invertible.
If $r \in I^{p-1}$, then $\frac{r}{p}\sum_i \sigma^i$ is an integral in $H(I)$ by \eqref{integral}, since $\zeta-1 \in I$. For the reverse inclusion,
observe that by construction we have $I^{p-1}=\{\alpha \in K:\alpha x^{p-1}\in H(I)\}$. Let $\int$ be an integral in $H(I)$.
There is $\lambda \in K$ such that $\int=\frac{\lambda}{p}\sum_i \sigma^i$. Then $\lambda \xi x^{p-1} \in H(I)$ by \eqref{integral} and thus $\lambda \in I^{p-1}$.
\epf

We will next prove that all Hopf orders of $KC_p$ are global Larson orders. Over a local ring, this is a theorem by Greither, see \cite[Lemma 1.2, page 40]{G}.
We will use the local to global result of Tate and Oort \cite[Lemma 4]{TO} to pass to the number field case. \par \smallskip

Let $\p \subset R$ be a prime ideal such that $p \in \p$. Consider the corresponding valuation $\nu$, scaled so that $\nu(p)=1$
(we find more convenient to write here the valuation in additive terms). Then it is easy to see that $\nu(1-\zeta) = \frac{1}{p-1}$ because $(\zeta-1)^{p-1} = (p).$

\begin{definition} \cite[Section 3]{L1}
Let $b\in R_{\p}$ be such that $0\leq \nu(b) \leq \frac{1}{p-1}$. Set $s=\nu(b)$. The {\it Larson order} $H(s)$ is the $R_{\p}$-subalgebra of $K_{\p}C_p$ generated by $\frac{1}{b}(\sigma-1)$.
\end{definition}

One can see, exactly as in Lemma \ref{glo}, that Larson orders are indeed Hopf orders, and that $H(s)$ does not depend on the choice of $b$.
Notice that $H(s)$ is defined only if there is an element with valuation $s$ in $R_{\p}$.
We have the following classification result by Greither, see \cite[Theorem 3.0.0]{U} and \cite[Lemma 1.2, page 40]{G}.

\begin{theorem}[Greither]
All Hopf orders of $K_{\p}C_p$ over $R_{\p}$ are Larson orders.
\end{theorem}

We recall the following result of Tate and Oort:

\begin{proposition} \cite[Lemma 4]{TO}\label{cartesiansq}
For any commutative ring $T$, let $E(T)$ denote the set of isomorphism classes of group schemes of order $p$ over $T$. Then, the square
$$\xymatrix{E(R) \ar[r]\ar[d] & \prod_{\p\in Spec(R)}E(R_{\p})\ar[d] \\
E(K)\ar[r] &  \prod_{\p\in Spec(R)}E(K_{\p})}$$
where the maps are given by extension of scalars, is cartesian.
\end{proposition}

With this in hand we can establish:

\begin{theorem}\label{descrip}
Every Hopf order of $KC_p$ over $R$ is a global Larson order.
\end{theorem}

\pf
A Hopf order $X$ of $KC_p$ over $R$ can be viewed as a group scheme of order $p$. Proposition \ref{cartesiansq} tells us that giving $X$ is the same as giving its extension of scalars to $K$ and $R_{\p}$ for every $\p \in \textrm{Spec}(R)$, in a compatible way. The extension of scalars of $X$ to $K$ will be just $KC_p$, and thus we know the extension of scalars to all $K_{\p}$. Furthermore, if $\p\in \textrm{Spec}(R)$ satisfies $p\notin \p$, then we only have one Hopf order over $R_{\p}$. This is because all primitive idempotents will be contained in any Hopf order. \par \smallskip

The different orders will differ only by their extension of scalars to $R_{\p}$ with $p\in\p$. We know by Greither's Theorem that $X \otimes_R R_{\p}$ is a Larson order over $R_{\p}$. Let $\q_1^{r_1}\cdots \q_l^{r_l}$ be the prime decomposition  of $(\zeta-1)$ in $R$. Assume that $X \otimes_R R_{\q_i}$\vspace{-1.5pt} is isomorphic to $H(s_i)$ over $R_{\q_i}$. Consider the ideal $I=\prod_i\q_i^{(p-1)r_is_i}$. One can now see that the Larson order $H(I)$ will give rise to exactly the same localizations as $X$ at $\q_i$. Since the square in Proposition \ref{cartesiansq} is cartesian, this means that $X=H(I)$.
\epf

We know how the integrals inside Larson orders look like by Lemma \ref{glo}. As a consequence:

\begin{corollary}
A Hopf order $H(J)$ of $KC_p$ over $R$ which contains $\frac{1}{p}I^{p-1}\sum_i \sigma^i$ contains the Hopf order $H(I)$.
\end{corollary}

\pf
Using the prime decomposition of ideals, $I^{p-1}\subseteq J^{p-1}$ implies $I\subseteq J$.
\epf

The computation of the submodule of integrals in Lemma \ref{glo} together with Theorem \ref{descrip} has the following outcome, from which we will derive the necessary condition in our main theorem:

\begin{corollary}\label{pisigmae}
Let $X$ be a Hopf order of $KC_p$.
\begin{enumerate}
\item[(i)] Suppose that the $R$-submodule of integrals of $X$ \vspace{-3pt} is generated by $\frac{1}{\sqrt{p}}\sum_i \sigma^i$.
Then there exists an ideal $I$ of $R$ such that $I^{2(p-1)}=(p)$. \vspace{3pt}
\item[(ii)] Suppose that $\frac{1}{\sqrt{p}}\sum_i \sigma^i \in X$ and there is $\pi \in K$ such that $\pi^2=\zeta-1$.\vspace{-2pt}
Then $\frac{1}{\pi}(\sigma-1) \in X$.
\end{enumerate}
\end{corollary}

\pf (i) In view of Theorem \ref{descrip}, $X$ is isomorphic to $H(I)$ for some ideal $I$ of $R$ containing $\zeta-1$.
By hypothesis and Lemma \ref{glo} the submodule of integrals is
$$R\bigg(\frac{1}{\sqrt{p}}\sum_i \sigma^i\bigg)= \frac{1}{p}I^{p-1}\sum_i \sigma^i.$$
Then $I^{p-1}=(\sqrt{p})$ and thus $I^{2(p-1)}=(p)$. \par \medskip

(ii) From the hypothesis and Lemma \ref{glo}, we obtain $(\sqrt{p}) \subseteq I^{p-1}$. We know that $(\zeta-1)^{p-1}=(p)$. Using the prime factorization of ideals, we have $(\pi)^{p-1}=(\sqrt{p}) \subseteq I^{p-1}$. This implies that $(\pi) \subseteq I$. Then the element
$\frac{\pi}{\zeta-1}(\sigma-1)=\frac{1}{\pi}(\sigma-1) \in X$ by the construction of $H(I)$.
\epf

\section{An explicit description of Nikshych's Hopf algebra}\label{Nikshych}
\setcounter{equation}{0}

The goal of this section will be to write in an explicit way Nikshych's Hopf algebra. \par \smallskip

For an odd prime number $p$, Nikshych constructed in \cite{N1} a finite-dimensional, semisimple, weakly group-theoretical and non  group-theoretical Hopf algebra $H_p$ of dimension $4p^2$. It was defined in terms of a tensor category and a fiber functor. The representation category $\Rep(H_p)$ is constructed from the representation category of another Hopf algebra, $A_p$, by means of equivariantization by $C_2$. As fusion categories, $\Rep(H_p) \simeq \Rep(A_p)^{C_2}$. The Hopf algebra $A_p$ first appeared in the work of Masuoka \cite{M1}.
The above equivalence implies that $H_p$ fits into the short exact sequence
$$K\rightarrow A_p\rightarrow H_p\rightarrow KC_2\rightarrow K.$$

To describe explicitly the structure of $H_p$ we need to write the structure of $A_p$, the action of the generator $g$ of $C_2$ on $A_p$, and the comultiplication of $g$. \par \smallskip

{\it From now on we abbreviate $A_p$ to $A$ and $H_p$ to $H$. In this section we assume that $K$ is algebraically closed of characteristic zero.} \par \smallskip

The main result of this section is the following:

\begin{theorem}\label{structurenikshych}
Let $\zeta \in K$ be a primitive $p$-th root of unity. The Hopf algebra $H$ is generated, as an algebra over $K$, by the elements $e_0,e_1,u_a,u_b,v_a,v_b$ and $g$ subject to the following relations:
$$\begin{array}{lllll}
e_0+e_1=1,       & \hspace{3mm} e_0e_1=e_1e_0=0,   &               &                                       & \vspace{3pt}  \\
u_a^p=u_b^p=e_0, & \hspace{3mm} e_0u_a=u_a,        &  e_0u_b=u_b,  & \hspace{3mm} u_au_b=u_bu_a,           & \vspace{3pt} \\
v_a^p=v_b^p=e_1, & \hspace{3mm} e_1v_a=v_a,        &  e_1v_b=v_b,  & \hspace{3mm} v_av_b = \zeta_p v_bv_a, & \vspace{3pt} \\
g^2=1,           & \hspace{3mm} gu_a= u_bg,        &  gu_b= u_ag,  & \hspace{3mm} gv_a = v_ag,             & \hspace{3mm} gv_b = v_bg. \\
\end{array}$$
The comultiplication, counit, and antipode of $H$ is given by the following formulas:
\begin{equation}\label{comultA}
\begin{array}{lll}
\Delta(u_a) = u_a\otimes u_a + v_a\otimes v_a,       & \hspace{5mm} \varepsilon(u_a)=1, & \hspace{5mm} S(u_a) = u_a^{p-1}, \vspace{3pt} \\
\Delta(u_b) = u_b\otimes u_b + v_b\otimes v_b^{p-1}, & \hspace{5mm} \varepsilon(u_b)=1, & \hspace{5mm} S(u_b)=u_b^{p-1}, \vspace{3pt} \\
\Delta(v_a) = u_a \otimes v_a + v_a\otimes u_a,      & \hspace{5mm} \varepsilon(v_a)=0, & \hspace{5mm} S(v_a) = v_a^{p-1}, \vspace{3pt} \\
\Delta(v_b) = u_b\otimes v_b + v_b\otimes u_b^{p-1}, & \hspace{5mm} \varepsilon(v_b)=0, & \hspace{5mm} S(v_b)=v_b.
\end{array}
\end{equation}
The comultiplication of $g$ is given by
\begin{equation}\label{comultg1}
\begin{array}{ll}
\Delta(g) = & \hspace{-6pt}  {\displaystyle \frac{1}{p^2}\sum_{i,j,k,l} \zeta^{kj-il} gu_a^i u_b^j \otimes gu_a^k u_b^l + \frac{1}{p}\sum_{k,l} \zeta^{-(k+l)k} gu_a^k u_b^l \otimes gv_a^{k+l} v_b^{k+l}} \vspace{4pt} \hspace{-7pt} \\
 & \hspace{-6pt}  {\displaystyle + \frac{1}{p} \sum_{k,l} \zeta^{k(k+l)} gv_a^{k+l}v_b^{(p-1)(k+l)} \otimes gu_a^k u_b^l + \frac{1}{p} \sum_{k,l} gv_a^k v_b^l \otimes gv_a^{(p-1)l}v_b^k.} \hspace{-7pt}
\end{array}
\end{equation}
The counit and antipode of $g$ are $\varepsilon(g)=1$ and $S(g)=g$.
\end{theorem}

The rest of this section will be devoted to prove Theorem \ref{structurenikshych}.

\subsection{The algebra $A$}\label{algA} As an algebra, $A$ is the direct sum
$$K(C_p\times C_p) \oplus K^c(C_p\times C_p),$$
where $c: (C_p \times C_p) \times (C_p \times C_p) \rightarrow K^{\times}$ is the $2$-cocycle given by
$$c(a^ib^j,a^kb^l)=\zeta^{-jk}, \qquad 0 \leq i,j,k,l <p.$$
Here $a,b$ are generators of $C_p\times C_p$. We present the group algebra $K(C_p\times C_p)$ by generators $u_a,u_b$ and defining relations $u_a^p=u_b^p=1, u_au_b=u_bu_a$. The twisted group algebra $K^c(C_p\times C_p)$ is presented by generators $v_a,v_b$ and relations $v_a^p=v_b^p=1, v_av_b=\zeta v_bv_a$. Notice that $K^c(C_p\times C_p)$ is isomorphic to the matrix algebra ${\rm M}_p(K)$. To shorten, {\it we set $A_0=K(C_p\times C_p)$ and $A_1=K^c(C_p\times C_p)$. We denote the units of $A_0$ and $A_1$ by $e_0$ and $e_1$ respectively.} So $1_A=e_0 + e_1$ and $e_0e_1=e_1e_0=0$.
Unless otherwise specified, the inverses are taking inside either $A_0$ or $A_1$. For example, $u_a^{-1}$ means $u_a^{p-1}$. \par \smallskip

The comultiplication, counit, and antipode of $A$ are described in \eqref{comultA} above.

\subsection{The algebra $H$} As an algebra, $H$ is the crossed product $A*KC_2,$ where $g$ acts as an algebra automorphism on $A$ by:
$$g(u_a)= u_b, \quad g(u_b) = u_a, \quad g(v_a) = v_a, \quad g(v_b) = v_b.$$
In $H$ we have the relations:
$$gu_a = u_bg, \quad gu_b = u_ag, \quad gv_a = v_ag, \quad gv_b = v_bg.$$
The hard part in the description of $H$ is the formula for $\Delta(g)$. Recall from \cite{N1} that $H$ is constructed as follows:
the automorphism $g$ induces an autoequivalence $$F:\Rep(A) \rightarrow \Rep(A),\ V \mapsto {}_g V.$$
Here ${}_gV=V$ as a vector space, with new action $x \cdot v=g(x)v$ for all $x \in A,v \in V$. The functor $F$ is a tensor equivalence. Moreover, $F^{-1}=F$. To compute $\Delta(g)$ we will need to describe the tensor structure of $F$. For this, we first need to consider the irreducible representations of $A$.

\subsection{Irreducible representations of $A$}\label{IrrA}
Every irreducible representation of $A$ is an irreducible representation of either $A_0$ or $A_1$. \par \smallskip

The algebra $A_0$ has $p^2$ one-dimensional irreducible representations, which we denote by $K^{i,j}$ with $0\leq i,j <p.$ As a $K$-vector space, $K^{i,j}=K$. The action of $u_a$ and $u_b$ on $K^{i,j}$ is given by:
$$u_a\cdot 1 = \zeta^i 1 \qquad u_b\cdot 1 =\zeta^j 1.$$

The algebra $A_1$ has only one irreducible representation, of dimension $p$,
which we denote by $M$. Let $\{m_i: 0 \leq i <p\}$ be a basis for $M$. The action of $A_1$ on $M$ is
$$v_a \cdot m_i = \zeta^i m_i, \qquad v_b \cdot m_i = m_{i+1}\  (\textrm{indices are taken}\  mod.\  p).$$

\subsection{Tensor structure on $F$}\label{tenstrF}
For any $V,W \in \Rep(A)$ irreducible we must establish an isomorphism $\theta_{V,W}:F(V\otimes W)\rightarrow F(V)\otimes F(W)$ satisfying the unit and associativity constraints. We do need to calculate these isomorphisms explicitly, as we will use them later to compute $\Delta(g)$. Observe that at the level of representations $F(K^{i,j}) = K^{j,i}$ and $F(M)=M$.

\subsubsection{Isomorphisms between certain representations of $A$}\label{isoms}
Given $x\in A_1$ invertible, ${}_xM$ stands for the following representation of $A_1$: as a vector space, ${}_xM=M$, and the action is given by
$$y\cdot m = (x^{-1}yx)m \qquad \forall y \in A_1,m \in M.$$
We have an isomorphism
$${}_x M \rightarrow M,\ m \mapsto xm.$$
This will be used in this subsection to define isomorphisms between certain tensor products of representation.
Consider the representation $K^{i,j} \otimes M$. Identify it with $M,$ as a vector space, via $1\otimes m \mapsto m$.
Under this identification, $v_a$ and $v_b$ act via $\zeta^iv_a$ and $\zeta^jv_b$ respectively. Since $v_av_b=\zeta v_bv_a$, we see that this is the same as ${}_xM$ for $x=v_a^{-j}v_b^i$.
Then we have an isomorphism of representations
\begin{equation}\label{isoijmleft}
l_{i,j}:K^{i,j}\otimes M\rightarrow M, \ 1\otimes m \mapsto (v_a^{-j}v_b^i)m.
\end{equation}
In a similar fashion, $M\otimes K^{i,j}$ is isomorphic to ${}_xM$ for $x=v_a^jv_b^i$ via
\begin{equation}\label{isoijmright}
r_{i,j}:M\otimes K^{i,j}\rightarrow M, \ m\otimes 1 \mapsto (v_a^jv_b^i)m.\smallskip
\end{equation}

We discuss separately the four different cases that occur in the description of $\theta_{V,W}$:

\subsubsection{Two representations of $A_0$}
We begin by considering the\vspace{0.5pt} case $V=K^{i,j}$ and $W=K^{k,l}$. We have $V\otimes W \simeq K^{i+k,j+l}.$
We must give an isomorphism\vspace{0.5pt} between $F(V\otimes W) \simeq F(K^{i+k,j+l}) \simeq K^{j+l,i+k}$ and\vspace{0.5pt} $F(V)\otimes F(W) \simeq K^{j,i} \otimes K^{l,k} \simeq K^{j+l,i+k}$.
It will be determined by a nonzero scalar $\mu((i,j),(k,l))$. Then:
$$\theta_{V,W}:F(V\otimes W) \rightarrow F(V)\otimes F(W), \ 1 \otimes 1 \mapsto \mu((i,j),(k,l))1 \otimes 1.$$
The associativity constraints yield that $\mu:(C_p\times C_p)^2\rightarrow K^{\times}$ is a $2$-cocycle. We shall compute $\mu$ explicitly in the sequel. We will see that:
\begin{equation}\label{thetaijkl}
\theta_{K^{i,j},K^{k,l}}:F(K^{i,j}\otimes K^{k,l}) \rightarrow F(K^{i,j})\otimes F(K^{k,l}), \ 1 \otimes 1 \mapsto \zeta^{il-jk} 1 \otimes 1.
\end{equation}

\subsubsection{One representation of $A_0$ and one representation of $A_1$}
We next consider the case $V=K^{i,j}$ and $W=M$ (and $V=M$ and $W=K^{i,j}$).
We first deal with the values $(i,j)=(0,1), (1,0)$ and then we will deduce a formula for an arbitrary pair $(i,j)$. \par \smallskip

We need to find an isomorphism between $F(K^{1,0}\otimes M)$ and $F(K^{1,0})\otimes F(M)$.
Both representations are isomorphic to $M$. Thus, up to a nonzero scalar, there is only one possible choice.
Using \eqref{isoijmleft}, we see that such an isomorphism must be given by
$$\theta_{K^{1,0},M}:F(K^{1,0}\otimes M)\rightarrow F(K^{1,0})\otimes F(M), \ 1\otimes m \mapsto \alpha_{1,0} \otimes (v_av_b)m,$$
for some $\alpha_{1,0} \in K$ (that will be determined later). In a similar fashion:
$$\begin{array}{ll}
\theta_{K^{0,1},M}:F(K^{0,1}\otimes M)\rightarrow F(K^{0,1})\otimes F(M), & 1\otimes m \mapsto \alpha_{0,1} \otimes (v_b^{-1}v_a^{-1})m, \vspace{2pt} \\
\theta_{M,K^{1,0}}:F(M\otimes K^{1,0})\rightarrow F(M) \otimes F(K^{1,0}), & m\otimes 1 \mapsto \beta_{1,0} (v_a^{-1}v_b) m \otimes 1, \vspace{2pt} \\
\theta_{M,K^{0,1}}:F(M\otimes K^{0,1})\rightarrow F(M)\otimes F(K^{0,1}), & m \otimes 1 \mapsto \beta_{0,1} (v_b^{-1}v_a) m\otimes 1,
\end{array}$$
for $\alpha_{0,1},\beta_{1,0},\beta_{0,1} \in K$. \smallskip

The tensor structure on $F$ will depend on $\alpha_{1,0},\alpha_{0,1},\beta_{1,0},\beta_{0,1}$, and $\mu$.
The compatibility of $F$ with associativity constraints will impose some restrictions on the possible values of \smallskip them.

We show by induction that the following formula holds for $(i,0)$ with $i \geq 2$:
$$\theta_{K^{i,0},M}:F(K^{i,0}\otimes M) \rightarrow F(K^{i,0}) \otimes F(M), \ 1\otimes m \mapsto \alpha_{1,0}^i \otimes (v_a^iv_b^i)m.$$
Using naturality and compatibility of $F$ with the associativity constraint we have the following commutative diagram:
$$\scalebox{0.95}[0.93]{\xymatrix@R=24pt@C=1pt{F(K^{i,0}\otimes M) \ar[d]_{\theta_{K^{i,0},M}} \ar[r] & F(K^{i-1,0}\hspace{-1.5pt}
\otimes\hspace{-1.5pt} K^{1,0}\hspace{-1.5pt} \otimes\hspace{-1.5pt} M) \ar[r]^(0.55){F(id\, \otimes\, l_{1,0})} & F(K^{i-1,0} \otimes M) \ar[d]^{\theta_{K^{i-1,0},M}} \\
F(K^{i,0}) \otimes F(M)    & & F(K^{i-1,0}) \otimes F(M) \ar[d]^{id\, \otimes\, F(l_{1,0}^{-1})} \\
F(K^{i-1,0}\hspace{-1.5pt} \otimes\hspace{-1.5pt} K^{1,0})\hspace{-1.5pt} \otimes\hspace{-1.5pt} F(M) \ar[u] &   & F(K^{i-1,0})\hspace{-1.5pt} \otimes\hspace{-1.5pt} F(K^{1,0}\hspace{-1.5pt} \otimes\hspace{-1.5pt} M) \ar[dl]^(0.4){\phantom{aaa} id\, \otimes\, \theta_{K^{1,0},M}} \\
  & F(K^{i-1,0})\hspace{-1.5pt} \otimes\hspace{-1.5pt} F(K^{1,0})\hspace{-1.5pt} \otimes\hspace{-1.5pt} F(M)  \ar[ul]^(0.6){\theta_{K^{i-1,0},K^{1,0}}^{-1} \otimes\, id\phantom{aaaa}} & }}$$
One can check that $1\otimes m$ is mapped to
$$\alpha_{1,0}^i\,\mu((i-1,0),(1,0))^{-1} \otimes (v_a^iv_b^i)m.$$
Without loss of generality, we can assume that $\mu((i,0),(j,0)) = \mu((0,i),(0,j)) = 1$, and then we arrive at the desired formula. \smallskip

By a similar calculation we also obtain:
$$\theta_{K^{0,j},M}:F(K^{0,j}\otimes M)\rightarrow F(K^{0,j})\otimes F(M),\hspace{5pt} 1\otimes m\mapsto \alpha_{0,1}^j \otimes (v_b^{-j}v_a^{-j})m.$$

We can combine these two isomorphisms with the associativity constraint to get the following general formula:
\begin{equation}\label{thetakijm}
\ \ \theta_{K^{i,j},M}\! : \! F(K^{i,j} \otimes M)\! \rightarrow\! F(K^{i,j})\otimes F(M), \hspace{2pt} 1 \otimes m \mapsto \alpha_{1,0}^i\alpha_{0,1}^j\zeta^{i(i-j)} \otimes v_b^{i-j}v_a^{i-j}m.\hspace{-5pt}
\end{equation}
This is done as follows. Using naturality and compatibility of $F$ with the associativity constraint we can construct the following commutative diagram:
$$\scalebox{0.95}[0.93]{\xymatrix@R=24pt@C=1pt{
F(K^{i,j} \otimes M) \ar[d]_{\theta_{K^{i,j},M}} \ar[r] & F(K^{i,0} \otimes K^{0,j} \otimes M) \ar[r]^(0.55){F(id\, \otimes\, l_{0,j})} & F(K^{i,0} \otimes M) \ar[d]^{\theta_{K^{i,0},M}} \\
F(K^{i,j}) \otimes F(M)    & & F(K^{i,0}) \otimes F(M) \ar[d]^{id\, \otimes\, F(l_{0,j}^{-1})} \\
F(K^{i,0} \otimes K^{0,j}) \otimes F(M) \ar[u] &   & F(K^{i,0}) \otimes F(K^{0,j} \otimes M) \ar[dl]^(0.4){\phantom{aaaa} id\, \otimes\, \theta_{K^{0,j},M}} \\
  & F(K^{i,0}) \otimes F(K^{0,j}) \otimes F(M)  \ar[ul]^(0.6){\theta_{K^{i,0},K^{0,j}}^{-1}\, \otimes\, id\phantom{aaaa}} & }}$$
Following the longest path, we obtain:
\begin{equation}\label{kijm}
\theta_{K^{i,j},M}(1 \otimes m)=\frac{\alpha_{1,0}^i\alpha_{0,1}^j\zeta^{i^2}}{\mu((i,0),(0,j))} \otimes v_b^{i-j}v_a^{i-j}m.
\end{equation}
We can write a similar diagram with $K^{0,j}\otimes K^{i,0}\otimes M$ in the upper central term and proceeding accordingly we get:
$$\theta_{K^{i,j},M}(1 \otimes m)=\frac{\alpha_{1,0}^i\alpha_{0,1}^j\zeta^{i^2-2ij}}{\mu((0,j),(i,0))} \otimes v_b^{i-j}v_a^{i-j}m.$$
These two equalities yield the following formula for $\mu$:
$$\frac{\mu((0,j),(i,0))}{\mu((i,0),(0,j))}=\zeta^{-2ij}.$$
Since $C_p \times C_p$ is abelian and $K$ is assumed to be algebraically closed of characteristic zero, this completely determines the cohomology class of $\mu$.
We choose the following representative from this cohomology class:
$$\mu((i,j),(k,l))=\zeta^{il-jk}.$$
Substituting this in \eqref{kijm} we arrive at the desired formula for $\theta_{K^{i,j},M}$.
By making this choice we also assure that $F^2=Id$ on the subcategory of representations \smallskip of $A_0$.

By a similar calculation, we obtain:
\begin{equation}\label{thetamkij}
\ \ \theta_{M,K^{i,j}} \! :\! F(M\otimes K^{i,j})\! \rightarrow\! F(M)\otimes F(K^{i,j}), m \otimes 1\mapsto \beta_{1,0}^i \beta_{0,1}^j\zeta^{j(j-i)}v_a^{j-i} v_b^{i-j}m \otimes 1.\hspace{-5pt}
\end{equation}

We have described so the tensor structure on $F$ for the tensor product of representations of $A_0$ with representations of $A_1$.
One can verify that {\it this structure is indeed compatible with all the associativity constraints involving two irreducible representations of $A_0$ if and only if
$\alpha_{1,0},\alpha_{0,1},\beta_{1,0},$ and $\beta_{0,1}$ are $p$-th roots of unity. Moreover, $F^2=Id$ on $K^{i,j} \otimes M$ and $M\otimes K^{i,j}$ if and only if
\begin{equation}\label{comp1}
\alpha_{1,0}\alpha_{0,1}=\beta_{1,0}\beta_{0,1}=1.
\end{equation}
We shall assume that this holds henceforth.}

\subsubsection{Two representations of $A_1$}
Lastly, we compute the isomorphism between $F(M\otimes M)$ and $F(M)\otimes F(M)$. We know that $M\otimes M \simeq \oplus_{i,j=0}^{p-1}\, K^{i,j}.$ One can easily check that the element $q_{i,j} \in M\otimes M$ spanning the 1-dimensional representation isomorphic to $K^{i,j}$ must be of the form
$$q_{i,j}=\lambda_{i,j}\sum_t \zeta^{-tj} m_t\otimes m_{i-t}, \ \textrm{with}\ \lambda_{i,j} \in K.$$
(Unless otherwise specified, throughout the limits in the sums are understood to run from $0$ to $p-1$.)
We take $\lambda_{i,j}=1$ for every $i,j$. The isomorphism is given by:
\begin{equation}\label{thetamm}
\theta_{M,M}:F(M\otimes M)\rightarrow F(M)\otimes F(M), q_{i,j} \mapsto \gamma_{i,j}q_{j,i},
\end{equation}
for some $\gamma_{i,j} \in K.$ Using naturality and compatibility of $F$ with the associativity constraint at $K^{i,j}\otimes M\otimes M$ we obtain the following commutative diagram:
$$\scalebox{0.95}[0.91]{
\xymatrix@R=24pt@C=1pt{
\bigoplus\limits_{s,t} F(K^{i,j} \otimes K^{s,t}) \ar[d]_{\bigoplus\limits_{s,t} \theta_{K^{i,j},K^{s,t}}} \ar[r] & F(K^{i,j}
\otimes M \otimes M) \ar[r]^(0.55){F(l_{i,j}\, \otimes\, id)} & F(M \otimes M) \ar[d]^{\theta_{M,M}} \\
\bigoplus\limits_{s,t} F(K^{i,j}) \otimes F(K^{s,t}) \ar[d]   &  &  F(M) \otimes F(M) \ar[d]^{F(l_{i,j}^{-1})\, \otimes\, id} \\
F(K^{i,j}) \otimes F(M  \otimes M) \ar[rd]_{id\, \otimes\, \theta_{M,M}} &   & F(K^{i,j} \otimes M) \otimes F(M) \ar[dl]^(0.4){\phantom{aaaa} \theta_{K^{i,j},M}\, \otimes\, id} \\
  & F(K^{i,j}) \otimes F(M) \otimes F(M) & }}$$
Through the isomorphism on the upper right side, $1 \otimes 1 \in K^{i,j} \otimes K^{s,t}$ is mapped to $\gamma_{s+i,t+j}\alpha_{1,0}^i\alpha_{0,1}^j\zeta^{it-js} \otimes q_{t,s}$.
Through the isomorphism on the left side, $1 \otimes 1$ is mapped to $\gamma_{s,t}\zeta^{it-js} \otimes q_{t,s}$. From here,
\begin{equation}\label{comp2}
\gamma_{i,j}=\alpha_{1,0}^{j-i}\gamma_{0,0}.
\end{equation}
By considering the associativity constraint for $M\otimes M \otimes K^{i,j}$ and writing the analogous diagram we get $\gamma_{i,j}=\beta_{1,0}^{j-i}\gamma_{0,0}$. This implies
\begin{equation}\label{comp3}
\alpha_{1,0}=\beta_{1,0}.
\end{equation}
The tensor structure of $F$ on $M\otimes M$ depends therefore on $\alpha_{1,0}$ (which is a $p$-th root of unity) and $\gamma_{0,0}$
(which equals $\pm 1$ since $F^2=Id$ on $M\otimes M$). \smallskip

By checking compatibility with all associativity constraints we see that the isomorphism we have constructed does furnish a tensor structure on $F$.
It can be shown directly that no matter what choice we make for $\gamma_{0,0}$ and $\alpha_{1,0}$, we will always end up with an isomorphic functor.
{\it We can thus assume, without loss of generality, that} $$\gamma_{0,0}=\alpha_{1,0}=1.$$ Then, the scalars $\alpha_{0,1}, \beta_{1,0}, \beta_{0,1},$ and $\gamma_{i,j}$ equal $1$ by equations \eqref{comp1}, \eqref{comp2}, and \eqref{comp3}. This finishes the description of the tensor structure on $F$. \smallskip

We summarize our discussion in the following result.

\begin{proposition}\label{uniquef}
Let $A$ be the Hopf algebra defined in Subsection \ref{algA}. Consider its irreducible representations $K^{i,j}$, with $0\leq i,j <p$, and $M$ defined in Subsection \ref{IrrA}. There exists (up to isomorphism) only one tensor functor $F:\Rep(A) \rightarrow \Rep(A)$ such that $F(K^{i,j}) \simeq K^{j,i}$ and $F(M)\simeq M$. It is given by the equations \eqref{thetaijkl}, \eqref{thetakijm}, \eqref{thetamkij}, and \eqref{thetamm}, where the scalars $\alpha_{1,0}, \alpha_{0,1}, \beta_{1,0}, \beta_{0,1},$ and $\gamma_{i,j}$ equal $1$.
\end{proposition}

\subsection{The comultiplication of $g$}
The category $\Rep(H)$ can be identified with that of $F$-equivariant representations of $A$ as follows:
if $V \in \Rep(H)$, then $V \in \Rep(A)$ by restriction, and $\tilde{g}:V\rightarrow V, v \mapsto gv$ establishes an isomorphism between $V$ and $F(V)$. \smallskip

We now consider the regular representation of $H$. The following diagram should be commutative:
$$\xymatrix{H\otimes H \ar[r]^(0.45){\Delta(g)\cdot }\ar[rd]_(0.42){(g\otimes g)\cdot} & F(H\otimes H)\ar[d]^{\Omega \cdot} \\ & F(H)\otimes F(H)}$$
where $\Omega$ comes from the tensor structure of $F$. Since $g=g^{-1}$, we have:
$$\Delta(g)=(g\otimes g)\Omega.$$
For $V,W \in \Rep(A)$ the isomorphism $\theta_{V,W}:F(V\otimes W)\rightarrow F(V)\otimes F(W)$ is given by multiplication by $\Omega \in A\otimes A$.
The reason for this is the following: the isomorphism $\theta_{A,A}: A\ot A \simeq F(A\ot A) \to F(A)\ot F(A) \simeq A\ot A$ is natural, and hence it must commute with multiplication from the right by elements of $A\ot A$. So, it must be given by multiplication from the left by some element $\Omega \in A\ot A$. The same holds for $V,W \in \Rep(A)$ by the naturality of $\theta$ again with respect to any morphisms $A \rightarrow V$ and $A\rightarrow W$. Then, the computation of $\Omega$ can be derived from our knowledge of these isomorphisms for any two irreducible representations of $A$.
To do this, we first need the decomposition of the regular representation of $A$ as a direct sum of irreducible representations.
For $i,j=0,\ldots, p-1$ let $f_{ij} \in A_0$ denote the idempotent upon which $u_a$ acts by $\zeta^i$ and $u_b$ by $\zeta^j.$ It is:
\begin{equation}\label{idemfij}
f_{ij} = {\displaystyle \frac{1}{p^2}\sum_{k,l} \zeta^{-(ik+jl)}u_a^ku_b^l}.
\end{equation}
Let $V_{ij}=A_0f_{ij}$. Then $V_{ij} \simeq K^{i,j}$. Consider in $A_1$ the element
$$h_i= {\displaystyle \frac{1}{p}\sum_{k} \zeta^{-ik}v_a^k}.$$
Let $W_i$ be the subspace spanned by $v_b^lh_i$ for $l=0,\ldots, p-1.$ Then $W_i \simeq M$ by mapping $v_b^{l-i}h_i$ to $m_l$. Thus we have:
$$A=\Big(\bigoplus\limits_{i,j} V_{ij}\Big) \bigoplus \Big(\bigoplus\limits_{i} W_i\Big).$$
We claim that:
\begin{equation}\label{descripJ}
\begin{array}{ll}
\Omega  = & {\displaystyle \frac{1}{p^2}\sum_{i,j,k,l} \zeta^{kj-il} u_a^i u_b^j \otimes u_a^k u_b^l + \frac{1}{p}\sum_{k,l} \zeta^{-(k+l)k} u_a^k u_b^l \otimes v_a^{k+l} v_b^{k+l}} \vspace{3pt} \\
 & {\displaystyle + \frac{1}{p} \sum_{k,l} \zeta^{k(k+l)} v_a^{k+l}v_b^{-(k+l)} \otimes u_a^k u_b^l + \frac{1}{p} \sum_{k,l} v_a^k v_b^l \otimes v_a^{-l}v_b^k.}
\end{array}
\end{equation}
Using \eqref{thetaijkl}, \eqref{thetamkij}, \eqref{thetakijm}, and \eqref{thetamm}, this formula for $\Omega$ is proved by checking directly the following equalities,
which we leave to the reader: \vspace{1pt}
$$\begin{array}{l}
\theta_{V_{ij},V_{kl}}(f_{ij} \otimes f_{kl})=\zeta^{il-jk} f_{ij} \otimes f_{kl} = \Omega(f_{ij} \otimes f_{kl}), \vspace{5pt} \\
\theta_{V_{ij},M}(f_{ij} \otimes v_b^{k-l}h_l)=\zeta^{(i-j)(k+i)}f_{ij} \otimes v_b^{k+i-j-l}h_l = \Omega(f_{ij} \otimes v_b^{k-l}h_l), \vspace{5pt} \\
\theta_{M,V_{ij}}(v_b^{k-l}h_l \otimes f_{ij})=\zeta^{(j-i)(k+i)}v_b^{k+i-j-l}h_l \otimes f_{ij} = \Omega(v_b^{k-l}h_l \otimes f_{ij}), \vspace{-2pt}
\end{array}$$
$$\begin{array}{l}
\hspace{10pt} {\displaystyle \theta_{M,M}\Big(\sum_{k} \zeta^{-jk} v_b^{k-l}h_l \otimes v_b^{i-k-l}h_l\Big) = \sum_k \zeta^{-ik} v_b^{k-l}h_l \otimes v_b^{j-k-l}h_l} \vspace{1pt} \\
\hspace{6.55cm} = {\displaystyle \Omega \Big(\sum_{k} \zeta^{-jk} v_b^{k-l}h_l \otimes v_b^{i-k-l}h_l}\Big). \vspace{4pt}
\end{array}$$
A careful calculation reveals that $S(g)=g$. This finishes the description of the Hopf algebra structure of $H$ and hence the proof of Theorem \ref{structurenikshych}.

\begin{remark}\label{defoverqzeta}
Although we used that $K$ is algebraically closed to reconstruct $H$, a posteriori we see from Theorem \ref{structurenikshych} that $H$ is defined over $\Q(\zeta)$.
\end{remark}

\section{Duality, (co)characters, and Hopf automorphisms}
\setcounter{equation}{0}

In this section we study further the structure of $H$: we describe its irreducible (co)representations and (co)characters, its Hopf automorphisms and we show that it is self-dual. The description of the (co)characters is one of the essential points in the proof of our main result since they provide elements in any Hopf order in view of Proposition \ref{character}. We keep the notation of the previous section.

\subsection{Dual Hopf algebra}
We present here the Hopf algebra structure of $H^*$. As a vector space, $H=A_0\oplus A_1\oplus gA_0\oplus gA_1.$ We consider the following basis of $H$:
\begin{equation}\label{basisH}
\Bas:=\{u_a^iu_b^j\}\cup\{v_a^iv_b^j\}\cup\{gu_a^iu_b^j\}\cup\{gv_a^iv_b^j\}.
\end{equation}
We denote the dual basis by:
\begin{equation}\label{dualbasisH}
\Bas^*:=\{s_{ij}\}\cup\{t_{ij}\}\cup \{\alpha_{ij}\}\cup\{\beta_{ij}\}.
\end{equation}
From \eqref{comultA} and \eqref{comultg1}, we easily see that $H=A \oplus gA$ as a coalgebra. Then
\begin{equation}\label{decompHdual}
H^*=A^* \oplus (gA)^*
\end{equation}
as an algebra. We denote by $\varepsilon_A$ and $\varepsilon_{gA}$ the counit of $H$ restricted to $A$ and $gA$ respectively. Then, $\varepsilon_A$ and $\varepsilon_{gA}$ are the central idempotents of $H^*$ giving the previous decomposition. The following result provides the full description of $H^*$.

\begin{proposition}\label{Hdual}
As an algebra, $H^*$ is the direct sum of the algebras $A^*$ and $(gA)^*$. The algebra $A^*$ is spanned by the elements $s_{ij}$ and $t_{ij}$ and its multiplication is given by:
\begin{equation}\label{multrules1}
\begin{array}{ll}
s_{ij}\hspace{0.5pt}s_{kl} = \delta_{i,k}\delta_{j,l}\,s_{ij},  &  \hspace{25pt} t_{kl}\hspace{0.5pt}s_{ij} = \delta_{i,k}\delta_{j,-l}\,t_{kl}, \vspace{5pt} \\
s_{ij}\hspace{0.5pt}t_{kl} = \delta_{i,k}\delta_{j,l}\,t_{kl}, &  \hspace{25pt} t_{ij}\hspace{0.5pt}t_{kl} = \delta_{i,k}\delta_{j,-l}\,s_{ij}.
\end{array}
\end{equation}
The algebra $(gA)^*$ is generated by the elements $\gamma_{ij}$ and $B$ subject to the following relations:
\begin{equation}\label{multrules2}
B^2=\varepsilon_{gA}, \hspace{20pt} \gamma_{ij}\hspace{0.5pt}\gamma_{kl} = \zeta^{il-jk}\hspace{0.5pt}\gamma_{i+k\:j+l}, \hspace{20pt} \textrm{ and } \hspace{15pt} B\hspace{0.5pt}\gamma_{ij}=\gamma_{ij}\hspace{0.5pt}B.
\end{equation}
The comultiplication, counit, and antipode of $H^*$ are given by:
\begin{equation}\label{comultHdual}
\begin{array}{l}
\Delta(s_{ij}) = {\displaystyle \sum_{k,l} s_{kl}\otimes s_{i-k\:j-l} + \frac{1}{p^2}\zeta^{-(il+jk)}\gamma_{kl}\otimes\gamma_{lk},} \vspace{3pt} \\
\hspace{1mm}\varepsilon(s_{ij})=\delta_{i,0}\delta_{j,0},  \hspace{15mm} S(s_{ij})= s_{-i\,-j}, \vspace{10pt} \\
\Delta(t_{ij}) = {\displaystyle \sum_{k,l} \zeta^{l(k-i)}t_{kl} \otimes t_{i-k\:j-l} + \frac{1}{p^2}\zeta^{-il}\gamma_{kl}B \otimes \gamma_{l-j\:k}B}, \vspace{3pt} \\
\hspace{1mm}\varepsilon(t_{ij}) = \delta_{i,0}\delta_{j,0}, \hspace{15mm} S(t_{ij})=\zeta^{-ij}t_{-ij}, \vspace{10pt} \\
\Delta(\gamma_{ij}) = {\displaystyle \sum_{k,l} \zeta^{li+kj}s_{kl}\otimes \gamma_{ij} + \zeta^{ki+lj}\gamma_{ij}\otimes s_{kl},} \vspace{3pt} \\
\hspace{1.5mm}\varepsilon(\gamma_{ij}) = 0, \hspace{23mm} S(\gamma_{ij}) = \gamma_{-j\,-i}, \vspace{10pt} \\
\Delta(B) =  {\displaystyle \sum_{k,l} \zeta^{kl}\gamma_{l0}B \otimes t_{kl} + t_{kl} \otimes \gamma_{0\,-l}B,} \vspace{3pt} \\
\hspace{1mm}\varepsilon(B)=0, \hspace{25mm}  S(B) = B.
\end{array}
\end{equation}
(The operations in the indices are all done modulo $p$.)
\end{proposition}

\pf
From the dual basis $\Bas^*$ in \eqref{dualbasisH}, we are going to construct a new basis of $H^*$ which is more convenient to express the multiplication. In $(gA_0)^*$, instead of $\{\alpha_{ij}\}$ we take the dual basis of $\{gf_{ij}\}$, where $\{f_{ij}\}$ are the idempotents in \eqref{idemfij}. We denote this basis by $\{\gamma_{ij}\}$. Then:
$$\gamma_{ij}(gu_a^ku_b^l)=\zeta^{ik+jl}.$$
The $s_{ij}$'s and $t_{ij}$'s form a basis of $A^*$ and the $\beta_{ij}$'s and $\gamma_{ij}$'s form one of $(gA)^*$. A direct and tedious calculation yields the following formulas:
$$\begin{array}{ll}
s_{ij}\hspace{0.5pt}s_{kl} = \delta_{i,k}\delta_{j,l}\,s_{ij},  &  \hspace{20pt} t_{kl}\hspace{0.5pt}s_{ij} = \delta_{i,k}\delta_{j,-l}\,t_{kl}, \vspace{3pt} \\
s_{ij}\hspace{0.5pt}t_{kl} = \delta_{i,k}\delta_{j,l}\,t_{kl}, &  \hspace{20pt} t_{ij}\hspace{0.5pt}t_{kl} = \delta_{i,k}\delta_{j,-l}\,s_{ij},  \vspace{5pt} \\
\gamma_{ij}\hspace{0.5pt}\beta_{kl} = \zeta^{j(l+j+k-i)}\beta_{k-i+j\;l-i+j}, & \hspace{20pt} \beta_{kl}\hspace{0.5pt}\gamma_{ij}=\zeta^{j(i+k-j-l)}\beta_{k+i-j\;l+j-i}, \vspace{5pt} \\
\gamma_{ij}\hspace{0.5pt}\gamma_{kl} = \zeta^{il-jk}\gamma_{i+k\:j+l}. &
\end{array}$$
This gives the statement for the multiplication in $A^*$. For the one in $(gA)^*$ we proceed as follows: consider the element
\begin{equation}\label{elemB}
B=\sqrt{p}\sum_k \beta_{k0}.
\end{equation}
It commutes with the $\gamma_{ij}$'s in view of the above formulas. A simple computation shows that $B^2=\varepsilon_{gA}$. Each $\beta_{ij}$ can be expressed as
$$\beta_{ij} = \frac{1}{p\sqrt{p}}\sum_k \zeta^{-ki} \gamma_{k-j\: k}B.$$
This can be verified directly by using the equality:
\begin{equation}\label{gammaijB}
(\gamma_{ij}B)(gv_a^kv_b^l)=\sqrt{p}\,\zeta^{jk}\delta_{l,j-i}.
\end{equation}
Then $\{\gamma_{ij}\} \cup \{\gamma_{ij}B\}$ is a basis of $(gA)^*$. We change our basis of $H^*$ again to
\begin{equation}\label{basisL}
\mathcal{L}:=\{s_{ij}\}\cup\{t_{ij}\}\cup \{\gamma_{ij}\} \cup \{\gamma_{ij}B\}.
\end{equation}
The multiplication of $H^*$ is then fully described on $\mathcal{L}$ by \eqref{multrules1} and \eqref{multrules2}. \par \smallskip

We next compute the formulas for the comultiplication of $H^*$ given in \eqref{comultHdual}. These formulas follow from direct calculations, just using the multiplication in $H$. The calculations do not present any special difficulty.  We briefly indicate how to proceed for $s_{ij}$ and leave the details and the other cases to the reader. The element $s_{ij}$ vanishes on $A_1, gA_0$ and $gA_1$. Since $A_0A_0=(gA_0)(gA_0)=A_0$ and $A_0A_1=A_1A_0=0$ no other kind of summands can occur in the right-hand side. Hence it suffices to evaluate $\Delta(s_{ij})$ at $u_a^k u_b^l \otimes u_a^m u_b^n$ and $gf_{kl} \otimes gf_{mn}.$
The coefficients of $s_{kl} \otimes s_{mn}$ and $\gamma_{kl}\otimes \gamma_{mn} $ must be respectively:
$$\langle s_{ij},\! (u_a^k u_b^l)(u_a^m u_b^n) \rangle \! =\! \delta_{i,k+m}\delta_{j,l+n} \ \hspace{12pt} {\rm and} \ \hspace{12pt} \langle s_{ij}, \! (gf_{kl})(gf_{mn}) \rangle \! =\!
\frac{1}{p^2}\zeta^{-(il+jk)}\delta_{k,n}\delta_{l,m}.$$

Finally, one can check with no effort that the counit and antipode are the ones given in \eqref{comultHdual}.
\epf

\subsection{Self-duality}
Nikshych proved in \cite[Proposition 5.2]{N1} that $H$ and $H^*$ are isomorphic as algebras. In this subsection we strengthen this result by the following proposition:

\begin{proposition}\label{selfduality}
The Hopf algebras $H$ and $H^*$ are isomorphic.
\end{proposition}

\pf
Let us begin by finding inside $H^*$ a Hopf subalgebra isomorphic to $A$. Set $d=\frac{p+1}{2}$. Consider the elements:
\begin{equation}\label{elemHd}
\bar{u}_a=\sum_{k,l} \zeta^{(k+l)d} s_{kl}, \qquad \bar{u}_b=\sum_{k,l} \zeta^{(k-l)d}s_{kl},\qquad \bar{v}_a=\gamma_{dd},\qquad \bar{v}_b=\gamma_{-dd}.
\end{equation}
Let $\bar{A}$ be the subalgebra generated by $\bar{u}_a,\bar{u}_b,\bar{v}_a,$ and $\bar{v}_b$. Using the multiplication rules \eqref{multrules1} and \eqref{multrules2} one easily checks that the assignment
$u_x \mapsto \bar{u}_x, v_x \mapsto \bar{v}_x$ for $x \in \{a,b\}$ establishes an algebra isomorphism $\Psi$ between $A$ and $\bar{A}$.
The elements corresponding to the central idempotents $e_0$ and $e_1$ in Subsection \ref{algA} are
\begin{equation}\label{unitcounit}
\varepsilon_A=\sum_{k,l} s_{kl} \hspace{20pt} \textrm{and} \hspace{20pt} \varepsilon_{gA}=\gamma_{00}.
\end{equation}
Notice that $\varepsilon_A+\varepsilon_{gA}=\varepsilon_H=1_{H^*}$. Using formulas \eqref{comultHdual} one can verify with a long but direct computation that the above isomorphism is actually an isomorphism of Hopf algebras. \par \smallskip

Consider finally the element
$$\bar{g}=B+\sum_{k,l} \zeta^{dkl}t_{kl}.$$
It can be shown that $\bar{g}^2=1_{H^*}$, conjugation by $\bar{g}$ stabilizes $\bar{A}$, and, by the above isomorphism, $\bar{g}$ acts on $\bar{A}$ as $g$ acts on $A$.
Moreover, one can show that $\Psi$ extends to a Hopf algebra isomorphism from $H$ to $H^*$ by defining $g \mapsto \bar{g}$. This finishes the proof. \epf

\begin{remark}
If $p=1$ mod. $4$, then $\sqrt{p} \in \Q(\zeta)$ and the above isomorphism is defined over $\Q(\zeta)$. Otherwise, it is not defined over $\Q(\zeta)$ but over $\Q(\zeta,\omega)$,
with $\omega$ a primitive fourth root of unity, and maps $B$ to $\omega B$. Consider $H$ as defined over $\Q(\zeta)$.
Then $B$ belongs to $H \otimes_{\Q(\zeta)} K$ but not to $H$ because $\sqrt{p} \notin \Q(\zeta)$ in this  case.
In fact, since the orbit of $B$ under the group of Hopf automorphisms of $H$ is $\{B,-B\}$, see Subsection \ref{Hopfaut},
it follows that an isomorphism between $H$ and $H^*$ cannot be defined over $\Q(\zeta)$.
The Hopf algebra $H^*$ will be a form of $H$ but not isomorphic to it over $\Q(\zeta)$.
\end{remark}

In the next two subsections we describe the irreducible representations of $H$ and $H^*$ and their characters, see \cite[Proposition 5.2]{N1},
which will be used to find the possible Hopf orders of $H$. \par \smallskip

\subsection{Characters of $H$}
We have the following irreducible representations of $H$ and corresponding characters:

\subsubsection{Dimension 1} There are $2p$ irreducible representations of $H$ of dimension $1$. They arise from the elements in $A_0$ that are $g$-invariant. For $i=0,\ldots,p-1$ we have the representation $V_i^{+}$ (resp. $V_i^{-}$), upon which $A_1$ acts trivially, $u_a^k u_b^l$ acts through the scalar $\zeta^{(k+l)i}$, and $g$ acts as $1$ (resp. $-1$). By using the previously chosen basis $\mathcal{L}$ of $H^*$ (see Equation \ref{basisL}) we can write the characters of these representations as:
\begin{equation}\label{Hdim1}
\chi_{V_i^{\pm}} = {\displaystyle \pm \gamma_{ii}+\sum_{k,l} \zeta^{(k+l)i} s_{kl}.}
\end{equation}

\subsubsection{Dimension 2} The irreducible representations of $H$ of dimension $2$ come from the $1$-dimensional representations of $A_0$ which are not $g$-invariant. Therefore, their orbits have two elements: $K^{i,j}$ and $K^{j,i}$ for $i\neq j$. Such representations are parameterized by pairs $(i,j)$ with $i<j$. We denote them by $W_{ij}$. There are $\frac{p(p-1)}{2}$ such representations. The elements $g$ and $u_a^k u_b^l$ act on $W_{ij}$ as the matrices
$$\left(\begin{array}{ccc}
0 & 1\\
1 & 0\end{array}\right)\quad \textrm{and} \quad
\left(\begin{array}{ccc}
\zeta^{ik+jl} & 0 \\
0 & \zeta^{il+jk}\end{array}\right)
$$
respectively, and $A_1$ acts trivially. The associated characters with respect to the basis $\mathcal{L}$ of $H^*$ are:
\begin{equation}\label{Hdim2}
\chi_{W_{ij}^{\phantom{+}}} = {\displaystyle \sum_{k,l} (\zeta^{ik+jl}+\zeta^{il+jk})s_{kl}.}
\end{equation}

\subsubsection{Dimension p} Finally, there are two irreducible representations of $H$ of dimension $p$. They arise from the $p$-dimensional representation $M$ of $A_1$, see Subsection \ref{IrrA}. We denote them by $M^+$ and $M^-$. They have basis $\{m_0,\ldots ,m_{p-1}\}$, the elements in $A_1$ act as $v_am_i =\zeta^i m_i, v_b m_i = m_{i+1}$ and $g$ acts as $\pm 1$. The elements of $A_0$ act trivially. The corresponding characters in the basis $\mathcal{L}$ of $H^*$ are:

\begin{equation}\label{Hdimp}
\chi_{M^{\pm}} = {\displaystyle pt_{00} \pm \frac{1}{\sqrt{p}}\sum_{i} \gamma_{ii} B.}
\end{equation}

\subsection{Characters of $H^*$}
To describe the irreducible representations of $H^*$ we will use the decomposition \eqref{decompHdual} expressing $H^*$ as the direct sum of algebras
$H^*=A^* \oplus (gA)^*$. We start with the irreducible representations of $A^*$. By the multiplication rules \eqref{multrules1}, $A^*$ is the direct sum of algebras
$$A^*= \Big(\bigoplus_{i} R_i\Big) \bigoplus \Big(\bigoplus_{i,j} R_{i,j}\Big),$$
where $R_i$ is spanned by $s_{i0}$ and $t_{i0}$ and $R_{ij}$ by $s_{ij},s_{i\,-j},t_{ij},t_{i\,-j}$.
The index $i$ runs from $0$ to $p-1$ and $j$ from $1$ to $\frac{p-1}{2}$ to avoid repetitions. \par \smallskip

\subsubsection{Dimension 1}\label{dim1} The algebra $R_i$ has two 1-dimensional representations, on both of which $s_{i0}$ acts as $1$ whereas $t_{i0}$ acts as $\pm 1$. We denote them by $L_i^{+}$ and $L_i^{-}$ respectively. The characters of these representations, expressed in the basis $\Bas$ of $H$, see Equation \ref{basisH}, are:
\begin{equation}\label{charHdual1}
\psi_{L_i^{\pm}} = {\displaystyle u_a^i \pm v_a^i.}
\end{equation}

\subsubsection{Dimension 2}\label{dim2} The algebra $R_{ij}$ is isomorphic to ${\rm M}_2(K)$. Therefore, it has one irreducible 2-dimensional representation, which we denote by $P_{ij}$. This representation is given by the following map:
\begin{equation*} s_{ij}\mapsto
\left(\begin{array}{ccc}
1 & 0\\
0 & 0 \end{array}\right), \hspace{10pt}
s_{i\: -j}\mapsto
\left(\begin{array}{ccc}
0 & 0\\
0 & 1 \end{array}\right), \hspace{10pt}
t_{ij} \mapsto
\left(\begin{array}{ccc}
0 & 1\\
0 & 0 \end{array}\right), \hspace{10pt}
t_{i\: -j}\mapsto
\left(\begin{array}{ccc}
0 & 0\\
1 & 0 \end{array}\right).
\end{equation*}
In the basis $\Bas$ of $H$ the characters of these representations are expressed as:
\begin{equation}\label{charHdual2}
\psi_{P_{ij}} = {\displaystyle u_a^iu_b^j + u_a^iu_b^{-j}.}
\end{equation}

\subsubsection{Dimension p}\label{dimp} Lastly, we discuss the irreducible representations of $(gA)^*$. Since $B^2=\varepsilon|_{gA}=1_{(gA)^*}$, we have the following two central idempotents:
$$\kappa=\frac{1}{2}(\varepsilon|_{gA}+B)\quad \textrm{and} \quad \kappa'=\frac{1}{2}(\varepsilon|_{gA}-B).$$
They induce the algebra decomposition
$(gA)^*=(gA)^*\kappa \oplus (gA)^*\kappa'$. From \eqref{multrules2} we obtain
$\gamma_{10}^p=\gamma_{01}^p=\varepsilon_{gA}$ and $\gamma_{10}\gamma_{01}=\zeta^2\gamma_{01}\gamma_{10}.$
Then $(gA)^*\kappa$ and $(gA)^*\kappa'$ are isomorphic to ${\rm M}_p(K)$.
Hence $(gA)^*$ has two $p$-dimensional irreducible representations,
which we denote by $N^+$ and $N^-$. Both have a basis $\{n_0,\ldots, n_{p-1}\}$ with actions
$$\gamma_{ij}n_l = \zeta^{ij+2il}n_{l+j},\qquad Bn_l=\pm n_l.$$

The characters of the above representations are given by:
\begin{equation}\label{charHdualp}
\psi_{N^{\pm}} = {\displaystyle \frac{1}{p}\sum_{i,j} gu_a^i u_b^j {\pm} \frac{1}{\sqrt{p}}\sum_i gv_a^i.}
\end{equation}

\subsection{Hopf automorphisms}\label{Hopfaut}
The group of Hopf automorphisms of $H$ is described by the following result:

\begin{proposition}
The group $Aut_{\scalebox{0.65}{Hopf}}\hspace{1pt}(H)$ is isomorphic to $C_2\times (C_2\ltimes C_p)$. Writing $C_2$ as $\{\pm 1\}$, the Hopf automorphism $\phi$ of $H$ corresponding to the triple $(\epsilon_1,\epsilon_2,t)$ is:
$$\begin{array}{lll}
\phi(u_a) = u_a^{\epsilon_2},  & \hspace{8mm} \phi(u_b) = u_b^{\epsilon_2},           &   \vspace{6pt}\\
\phi(v_a) = v_a^{\epsilon_2},  & \hspace{8mm} \phi(v_b) = \zeta^{t} v_b^{\epsilon_2}, &  \hspace{8mm} \phi(g)=g(e_0+\epsilon_1e_1).
\end{array}$$
\end{proposition}

\pf
We know from \eqref{comultA} and \eqref{comultg1} that $H=A \oplus gA$ as coalgebras and hence $H^*=A^* \oplus (gA)^*$ as algebras. The algebra $A^*$ splits as a direct sum of matrix algebras over $K$ of dimension $1$ or $4$ (Subsections \ref{dim1} and \ref{dim2}). On the other hand, the algebra $(gA)^*$ is the direct sum of two matrix algebras of dimension $p^2$ (Subsection \ref{dimp}). Let $\sigma \in Aut_{\scalebox{0.65}{{\it Hopf}}}\hspace{1pt}(H)$. Since $\sigma$ must preserve the Wedderburn decomposition of $H^*$, it must hold that $\sigma(A) \subseteq A$. Thus $\sigma\vert_A$ is a Hopf automorphism of $A$. We are so led to compute $Aut_{\scalebox{0.65}{{\it Hopf}}}\hspace{1pt}(A)$. This gives a group morphism
$$\Theta:Aut_{\scalebox{0.65}{{\it Hopf}}}\hspace{1pt}(H)\rightarrow Aut_{\scalebox{0.65}{{\it Hopf}}}\hspace{1pt}(A),\ \sigma \mapsto \sigma \vert_A.$$
Using this morphism, we are going to compute $Aut_{\scalebox{0.65}{{\it Hopf}}}\hspace{1pt}(H)$ in two steps: \par \medskip

{\it Step 1. Hopf automorphisms of $A$.} We know from Subsection \ref{algA} that $A$ has an algebra decomposition $A=A_0 \oplus A_1,$ where $A_0=K(C_p\times C_p)$ and $A_1=K^c(C_p\times C_p)$. Considering, as before, the dimensions of the simple components of the Wedderburn decomposition of $A_0$ and $A_1$ we get $\sigma(A_0)=A_0$ and $\sigma(A_1)=A_1$.
The group-like elements of $A$ are $u_a^i\pm v_a^i$ with $0 \leq i<p.$ Since $\sigma$ preserves group-like elements and the relations
$u_a^p=e_0$ and $v_a^p=e_1$, we must have $\sigma(u_a+v_a) = u_a^r + v_a^r$ for some
$r \neq 0$. As $\sigma(u_a)\in A_0$ and $\sigma(v_a) \in A_1$, we obtain
\begin{equation}\label{sigmauava}
\sigma(u_a) = u_a^r \hspace{8mm} \textrm{and} \hspace{8mm} \sigma(v_a) = v_a^r.
\end{equation}
On the other hand, $\sigma(u_b)=u_a^k u_b^s$ for some $k,s \neq 0$ because $\sigma$ induces a Hopf automorphism on the quotient Hopf algebra $A_0$ of $A$.
We derive that $k=0$ from the equality $\mu\Delta\sigma(u_b)=\sigma\mu\Delta(u_b)$. Here $\mu$ stands for the multiplication of $H$. So $\sigma(u_b)=u_b^s$.
Using the equality $\Delta\sigma(u_b)=(\sigma \otimes \sigma)\Delta(u_b)$ we arrive to $\sigma(v_b)=\lambda v_b^s$ for some $\lambda \in K^{\times}$.
Moreover, $\lambda^p=1$ because $\sigma(v_b)^p=e_1$. Put $\lambda=\zeta^t$ with $0 \leq t <p$.
Applying $\sigma$ to the relation $v_av_b = \zeta v_bv_a$ we get $sr=1$ mod.\! $p$. \linebreak Then
\begin{equation}\label{sigmaubvb}
\sigma(u_b) = u_b^s \hspace{8mm}  \textrm{and} \hspace{8mm}  \sigma(v_b) = \zeta^{t} v_b^s, \ \textrm{with} \ s=r^{-1}\ \textrm{mod.} \ p.
\end{equation}
Thus $\sigma$ determines a pair $(r,t) \in C_p^{\times} \times C_p.$
Conversely, one can check that any such a pair together with \eqref{sigmauava} and \eqref{sigmaubvb} defines a Hopf automorphism of $A$.
Finally, by composing two automorphisms one sees that $Aut_{\scalebox{0.65}{{\it Hopf}}}\hspace{1pt}(A) \simeq C_p^{\times} \ltimes C_p$.  \par \medskip

{\it Step 2. Computing the kernel and image of $\Theta$.} We claim that $\Ker \Theta \simeq C_2$. Let $\nu \in \Ker \Theta$. We know that $H$ has a coalgebra decomposition $H=A \oplus gA,$ that $\nu$ must preserve. Then $\nu(g) = gz$ for some $z\in A$. Since $\nu \vert_A =id_A$, we have for every $x\in A$:
$$gxg^{-1} = \nu(gxg^{-1}) = gzxz^{-1}g^{-1}.$$
From this it follows that $z \in Z(A)$. Recall that $\Delta(g)=(g\otimes g)\Omega$, where $\Omega$ is given in Equation \ref{descripJ}. Using this and that $\nu$ is a coalgebra map we get:
$$(g\otimes g)\Omega\Delta(z)=\Delta(gz)=\Delta\nu(g)=(\nu\otimes\nu)\Delta(g)=(gz\otimes gz)\Omega.$$
We also used here that $\Omega \in A \otimes A$ and $\nu \vert_A =id_A$.
Since $z \in Z(A)$ and $\Omega$ and $g$ are invertible, the above equality implies that $z$ is a group-like element of $A$.
As $1=\nu(g)^2=gzgz$, the only nontrivial option is $z=e_0-e_1$. Conversely, one can easily check that a map of this form defines an element of order $2$ in $\Ker \Theta$. \par \smallskip

We claim now that $\Ima \Theta \simeq C_2 \ltimes C_p$. Let $\sigma \in \Ima \Theta.$
Assume that $\sigma$ is given by $(r,t) \in C_p^{\times} \ltimes C_p$ and equations \eqref{sigmauava} and \eqref{sigmaubvb}.
Then, arguing as before, $\sigma(g)=gz$ for some $z\in A$. We have:
$$u_b^{r^{-1}}=\sigma(u_b)=\sigma(gu_ag^{-1})=gzu_a^{r}z^{-1}g^{-1}=u_b^{r}.$$
From this, $r^2 = 1$ mod. $p$ and so $r=\pm 1.$ Conversely, the Hopf automorphism $\tau$ of $A$ corresponding to $(1,t)$ is given by conjugation by the group-like element $u_a^t+v_a^t$.
Conjugation by the same element defines $\bar{\tau} \in Aut_{\scalebox{0.65}{{\it Hopf}}}\hspace{1pt}(H)$ such that $\Theta(\bar{\tau})=\tau$. Let $\varphi \in Aut_{\scalebox{0.65}{{\it Hopf}}}\hspace{1pt}(A)$ be corresponding to $(-1,0)$.
One can check effortless that $\varphi \in \Ima \Theta$ with preimage $\bar{\varphi}$ defined by
$\bar{\varphi}\vert_A =\varphi$ and $\bar{\varphi}(g)=g$. \par \smallskip

Thus we have a short exact sequence
$$1\rightarrow C_2 \rightarrow Aut_{\scalebox{0.65}{{\it Hopf}}}\hspace{1pt}(H) \rightarrow C_2 \ltimes C_p \rightarrow 1.$$
This sequence splits because $\bar{\varphi}$ has order $2$. The action on $C_2$ is trivial (this is the only possible action), and then
$$Aut_{\scalebox{0.65}{{\it Hopf}}}\hspace{1pt}(H) \simeq C_2 \times (C_2 \ltimes C_p).$$
\epf

\section{Orders of Nikshych's Hopf algebra}\label{Nik}
\setcounter{equation}{0}

In this section we will use the results of the previous sections to classify the orders of Nikshych's Hopf algebra.
We will see that Nikshych's Hopf algebra admits at most one order over any number field. \par \smallskip

We keep the conventions and notations of Section \ref{Nikshych}: $\zeta$ is a primitive $p$-th root of unity; $K$ is a number field containing $\zeta$; $R=\Oint_K$ is the ring of integers of $K$; $H$ denotes Nikshych's Hopf algebra of dimension $4p^2$, and $A$ stands for Masuoka's Hopf algebra of dimension $2p^2$, both defined over $K$. \par \smallskip

Recall from Remark \ref{defoverqzeta} that $H$ is defined over $\Q(\zeta)$. However, we will prove here that $H$ does not have orders over $\Oint_{\Q(\zeta)}$, but only over the ring of integers of some extension of $\Q(\zeta)$. Set $K=\Q(\zeta,\omega),$ where $\omega$ is a primitive fourth root of unity. The field $\Q(\zeta)$ contains either $\sqrt{p}$ or $\sqrt{-p}$, depending on the value of $p$ mod. $4$. The existence of $\omega$ allows us to assume that $\sqrt{p} \in K$ and treat our computations in a unified way avoiding the distinction of cases. \par \smallskip

The proof of Theorem \ref{thintro2} is quite involved. We will divide it into several parts.

\subsection{Elements that must be in any Hopf order}
Suppose that $X$ is a Hopf order of $H$ over $R$. Our goal in this first part is to prove that several elements of $H$, arising from (co)characters, must belong to $X$. This will be used later to show that all basis elements of $H$, given in \eqref{basisH}, must be in $X$. \par \smallskip

We retain the notation of Section \ref{Nikshych}: $e_0,e_1$ are the units of $A_0$ and $A_1$ and $\varepsilon_A,\varepsilon_{gA}$ denote the counits of $A$ and $gA$ respectively. We start with the following:

\begin{lemma}\label{e0e1}
The elements $e_0,e_1$ are in $X$ and $\varepsilon_A,\varepsilon_{gA}$ are in $X^{\star}$.
\end{lemma}

\pf We first show that $e_0,e_1 \in X$. The subalgebra $H_b$ of $H$ generated by $u_b$ and $v_b$ is a Hopf subalgebra.
Consider the algebra maps $\sigma:H_b \rightarrow K, u_b \mapsto \zeta, v_b \mapsto 0$ and $\tau:H_b \rightarrow K, u_b \mapsto 0, v_b \mapsto \zeta$.
They are group-like elements of $H_b^*$ and  $\sigma^2=\tau^p=1$ and $\sigma\tau=\tau^{p-1}\sigma$.
Then $H_b^* \simeq K(C_2 \ltimes C_p)$ as Hopf algebras and $X \cap H_b$ may be viewed as a Hopf order of $K(C_2 \ltimes C_p)^*$ by Proposition \ref{subsquo}(iii). According to the proof of \cite[Proposition 2.1]{CM}, $X \cap H_b$ contains the idempotents $t_0,t_1$ (notation as there).
Let $\{\nu_{\sigma^i\tau^j}\}_{i,j} \subset K(C_2 \ltimes C_p)^*$ be the dual basis of $\{\sigma^i\tau^j\}_{i,j}$.
Recall that $t_0=\sum_j \nu_{\tau^j}$ and $t_1=\sum_j \nu_{\sigma\tau^j}$.
One can verify directly\vspace{-1pt} that $\nu_{\tau^j}=\frac{1}{p}\sum_k \zeta^{-jk}u_b^k$ \linebreak and $\nu_{\sigma\tau^j}=\frac{1}{p}\sum_k \zeta^{(j-1)k} v_b^k$. Then $t_0=e_0$ and $t_1=e_1$. \par \smallskip

For the second statement, take into account that $H$ is self-dual by Proposition \ref{selfduality}. The isomorphism between $H$ and $H^*$ established there maps $e_0,e_1$ to $\varepsilon_A,\varepsilon_{gA}$ respectively, see \eqref{unitcounit}. We now get that $\varepsilon_A,\varepsilon_{gA}\in X^{\star}$ from self-duality of $H$, the above fact, and the first statement applied to $X^{\star}$ and $H^*$. \epf

Recall from \eqref{elemB} the element $B$ used in describing $H^*$.

\begin{lemma}\label{ge1B}
The elements $ge_1$ and $B$ belong to $X$ and $X^{\star}$ respectively.
\end{lemma}

\pf
We first prove that $ge_1 \in X.$ We know from Proposition \ref{character} that characters of $H^*$ are in $X$ and characters of $H$ are in $X^{\star}$.
Using the previous lemma, \eqref{charHdualp} and \eqref{Hdimp} we obtain that
\begin{align}
\Gamma_1 := e_0 \psi_{N^+} = {\displaystyle  \frac{1}{p} \sum_{i,j} gu_a^iu_b^j} \in X, \notag \hspace{13pt} \\
\Gamma_2 := \varepsilon_{gA}\chi_{M^+} = {\displaystyle \frac{1}{\sqrt{p}}\sum_{k} \gamma_{kk}B}\in X^{\star}. \label{elemW2}
\end{align}
Then $(\Gamma_2\otimes_R id_X)\Delta(\Gamma_1) \in X$. We check that $(\Gamma_2\otimes_R id_X)\Delta(\Gamma_1)=ge_1.$ Recall that $\Gamma_2$ vanishes on $A_0 \oplus A_1 \oplus gA_0$, so we only need to compute the part of $\Delta(\Gamma_1)$ in $gA_1 \otimes gA_1$. It is:
$$\begin{array}{l}
{\displaystyle \frac{1}{p^2}\sum_{i,j,k,l} (gv_a^kv_b^l \otimes gv_a^{-l}v_b^k)(v_a^iv_b^j\otimes v_a^iv_b^{-j})} \vspace{1pt} \\
\hspace{1.2cm} = {\displaystyle \frac{1}{p^2}\sum_{i,j,k,l} \zeta^{-i(k+l)}gv_a^{k+i}v_b^{l+j} \otimes gv_a^{i-l}v_b^{k-j}} \vspace{2pt} \\
\hspace{1.2cm} = {\displaystyle \frac{1}{p^2}\sum_{i',j',k,l'} \zeta^{(l'-i')(i'-k)}gv_a^{i'}v_b^{j'} \otimes gv_a^{l'}v_b^{i'-j'-l'}} \\
\hspace{6cm} {\small \textrm{putting} \ i'=k+i, j'=l+j,\ \textrm{and}\ l'=i-l,} \\
\hspace{1.2cm} = {\displaystyle \frac{1}{p}\sum_{i,j,l} \zeta^{(l-i)i}\Big(\frac{1}{p}\sum_{k} \zeta^{(i-l)k}\Big)gv_a^{i}v_b^{j} \otimes gv_a^{l}v_b^{i-j-l}} \\
\hspace{6cm} {\small \textrm{putting} \ i=i', j=j',\ \textrm{and}\ l=l',} \\
\hspace{1.2cm} = {\displaystyle \frac{1}{p}\sum_{i,j} gv_a^{i}v_b^{j} \otimes gv_a^{i}v_b^{-j}}. \vspace{3pt} \\
\end{array}$$
Applying $\Gamma_2 \otimes_R id_X$ to this expression we get
$$\begin{array}{ll}
{\displaystyle \frac{1}{p\sqrt{p}} \sum_{i,j,k} (\gamma_{kk}B)(gv_a^iv_b^j)gv_a^{i}v_b^{-j}} & \stackrel{\eqref{gammaijB}}{=} \ {\displaystyle \frac{1}{p}\sum_{i,k} \zeta^{ik} gv_a^i} \vspace{3pt} \\
 & \hspace{7pt} = \hspace{3pt} {\displaystyle \sum_{i} \Big(\frac{1}{p}\sum_{k}\zeta^{ik}\Big) gv_a^i} \vspace{3pt} \\
 & \hspace{7pt} = \hspace{3pt} ge_1.
\end{array}$$
Therefore $ge_1 \in X$. \par \medskip

We next show that $B\in X^{\star}.$ From \eqref{Hdimp} and Proposition \ref{character},
we know that $\chi_{M^+}=pt_{00} + \frac{1}{\sqrt{p}}\sum_{i} \gamma_{ii} B \in X^{\star}.$ Using Lemma \ref{e0e1}, we obtain $\varepsilon_A \chi_{M^+} = pt_{00}\in X^{\star}.$ Now,
\begin{equation}\label{proofB}
(\varepsilon_{gA}  \otimes \varepsilon_{gA})\Delta(pt_{00}) \stackrel{\eqref{comultHdual}}{=} \frac{1}{p} \sum_{k,l} \gamma_{kl} B \otimes \gamma_{lk}B \in X^{\star} \otimes_R X^{\star}.
\end{equation}
On the other hand, by \eqref{charHdualp} and Proposition \ref{character}, we have
$$\psi_{N^+}=\frac{1}{p}\sum_{i,j} gu_a^i u_b^j + \frac{1}{\sqrt{p}}\sum_{i} gv_a^i \in X.$$
Using again Lemma \ref{e0e1}, we get
$$e_1\psi_{N^+}=\frac{1}{\sqrt{p}}\sum_{i} gv_a^i \in X.$$
Finally, applying $e_1\psi_{N^+} \otimes_R id_{X^{\star}}$ to \eqref{proofB} we obtain
$$\frac{1}{p\sqrt{p}}\sum_{i,k,l} (\gamma_{kl}B)(gv_a^i)\gamma_{lk}B =\frac{1}{p}\sum_{i,k,l} \zeta^{il} \delta_{l-k,0}\gamma_{lk}B=\gamma_{00}B=B.$$
So, $B\in X^{\star}$.
\epf

\begin{lemma}\label{uava}
The elements $u_a, v_a, \frac{1}{\sqrt{p}} \sum_{i} u_a^i,$ and $\frac{1}{\sqrt{p}} \sum_{i} v_a^i$ belong to $X$.
\end{lemma}

\pf By \eqref{charHdual1} and Proposition \ref{character}, $u_a+v_a \in X$. Then $e_1(u_a+v_a)=v_a \in X$ and  $u_a=(u_a+v_a)-v_a \in X$. \vspace{-1.5pt}
We have just seen in the above proof that $\frac{1}{\sqrt{p}} \sum_{i} gv_a^i \in X$. Multiplying by $ge_1$, we have $\frac{1}{\sqrt{p}} \sum_{i} v_a^i \in X$.
Let $H_a$ be the Hopf subalgebra of $H$ generated by $u_a$ and $v_a$. Proposition \ref{subsquo}(iii) entails that $X \cap H_a$ is a Hopf order of $H_a$. Then
$$\Delta\Big(\frac{1}{\sqrt{p}} \sum_{i} v_a^i\Big) \stackrel{\eqref{comultA}}{=} \frac{1}{\sqrt{p}} \sum_{i} u_a^i \otimes v_a^i+v_a^i \otimes u_a^i \in (X \cap H_a) \otimes_R (X \cap H_a).$$
Consider the character $\varphi$ of $H_a$ given by $\varphi(u_a)=0$ and $\varphi(v_a)=1.$ By Proposition \ref{character}, $\varphi \in (X \cap H_a)^{\star}$.
Applying $\varphi \otimes_R id_{X \cap H_a}$ to the above equality we conclude that $\frac{1}{\sqrt{p}} \sum_{i} u_a^i \in X.$
\epf

\subsection{A special case}
If we show that $ge_0 \in X$, then it will follow from Lemmas \ref{e0e1}, \ref{ge1B}, and \ref{uava}, that all elements of the basis $\Bas$ in \eqref{basisH} of $H$ will be in any Hopf order $X$. Unlike for other elements, this can not be shown directly.
The strategy will be to adjoin to $K$ an element $\pi$ such that $\pi^2=\zeta-1$, prove the statement in this case and then derive it for $K$.
So, in this subsection we assume that $K$ contains such an element $\pi$. The proof requires some preparations.

\begin{lemma}\label{mapT}
The map $T:A_1 \rightarrow gA_0, v_a^iv_b^j \mapsto (B \otimes_R id_X)\Delta(gv_a^iv_b^j)$ can be expressed as
$$T(v_a^iv_b^j)=\frac{1}{\sqrt{p}} \sum_{k} \zeta^{jk} gu_a^k u_b^{i-k}.$$
Moreover, $T(X \cap A_1) \subseteq X \cap (gA_0).$
\end{lemma}

\pf Since $B$ vanishes on $A_0 \oplus A_1 \oplus gA_0$, only the part of $\Delta(gv_a^iv_b^j)$ in $gA_1 \otimes gA_0$ is relevant for the computation. We have:
$$\begin{array}{ll}
T(v_a^iv_b^j) & \hspace{-6pt}\stackrel{\textrm{Th.}\hspace{2pt} \ref{structurenikshych}}{=}  \hspace{-1pt}{\displaystyle \frac{1}{p} \sum_{k,l} \zeta^{k(k+l)} B\big(gv_a^{k+l}v_b^{-(k+l)}v_a^iv_b^j\big)gu_a^{k+i}u_b^{l-j}} \vspace{3pt} \\
              & \hspace{3pt} = {\displaystyle \frac{1}{p} \sum_{k,l} \zeta^{(k+i)(k+l)} B\big(gv_a^{k+l+i}v_b^{-(k+l)+j}\big)gu_a^{k+i}u_b^{l-j}} \vspace{3pt} \\
              & \hspace{-1pt} \stackrel{\eqref{gammaijB}}{=} {\displaystyle \frac{1}{\sqrt{p}} \sum_k \zeta^{(k+i)j}gu_a^{k+i}u_b^{-k}} \vspace{3pt} \\
              & \hspace{4pt} = {\displaystyle \frac{1}{\sqrt{p}} \sum_k \zeta^{kj}gu_a^{k}u_b^{i-k}}. \vspace{3pt} \\
\end{array}$$
Let now $x \in X \cap A_1$. By Lemma \ref{ge1B}, we know that $ge_1 \in X$ and $B \in X^{\star}$. Then $gx =ge_1x \in X$ and $\Delta(gx) \in X \otimes_R X$.
From here, $T(x)=(B \otimes_R id_X)\Delta(gx) \in X$.
\epf

\begin{proposition}\label{geomseries}
Let $Z$ be an $R$-algebra and $z,e \in Z$. Assume that $ze=ez=z$. Set $\tilde{z}=\frac{1}{\pi}(z-e).$ If $\tilde{z} \in Z$, then
$$\frac{1}{\sqrt{p}}\sum_i z^i$$
is an $R$-linear combination of powers of $\tilde{z}$.
\end{proposition}

\pf Set
$$\frac{(\pi \tilde{z} +e)^p - e}{\pi \tilde{z}}=\sum_{k=1}^{p} \binom{p}{k}(\pi \tilde{z})^{k-1}.$$
As in the proof of Lemma \ref{glo}, the fractional expression is just symbolic. The left-hand side equals $\sum_{i=0}^{p-1} z^i$.
We obtain the result by dividing this equation by $\sqrt{p}$, noticing that $\pi^{p-1}=\xi\sqrt{p}$ for some invertible $\xi\in R$,
and $\binom{p}{k}$ is divisible by $p$ for any $k=1,\ldots,p-1$. \epf

We are now ready to tackle the difficult point.

\begin{lemma}\label{ge0}
The element $ge_0$ belongs to $X$.
\end{lemma}

\pf View $A$ as a Hopf subalgebra of $H$ and $A_0$ as a quotient Hopf algebra of $A$ via projecting any element on its component in $A_0$.
Then $X \cap A_0$ is a Hopf order of $A_0$ in light of Proposition \ref{subsquo}. Look now at the Hopf subalgebra of $A_0$ generated by $u_a$.
Lemma \ref{uava} shows that $\frac{1}{\sqrt{p}} \sum_{i} u_a^i \in X.$ Applying Corollary \ref{pisigmae}(ii),\vspace{-2pt} we have $\frac{1}{\pi}(u_a-e_0) \in X \cap A_0$. \par \smallskip

On the other hand, Lemmas \ref{e0e1} and \ref{mapT} yield that
$$T(e_1)=\frac{1}{\sqrt{p}}\sum_{k} gu_a^ku_b^{-k} \in X.$$
Put $e=\frac{1}{p}\sum_{k} u_a^ku_b^{-k}$. Observe that $e$ is an idempotent\vspace{-1.5pt} and $T(e_1)=\sqrt{p}\, ge \in X$.
Let $G$ be the group generated by $\sigma,\tau$ subject to $\sigma^2=\tau^p=1, \sigma\tau=\tau\sigma$.
The assignments $e_1 \mapsto 0; u_a,u_b \mapsto \tau; g \mapsto \sigma$ define a surjective algebra map $f:H \rightarrow KG$.
It is easy to check that $f$ is a Hopf algebra map and $\Ker f$ equals the ideal generated by $e_1$ and $u_au_b^{-1}-e_0$.
By Proposition \ref{subsquo}(iv), $f(X)$ is a Hopf order of $KG$. The element $\sigma$ must be in $f(X)$ because it can be received from characters of $(KG)^*$.
Take $x \in X \cap A_0$ such that $f(x)=\sigma$. Then $x-ge_0=h(u_au_b^{-1}-e_0)$ for some $h \in H$. Multiplying by $\sqrt{p}\, ge$ we arrive to $\sqrt{p}\, (xge-e)=0$.
Thus $\sqrt{p}\, e= x(\sqrt{p}\,ge) \in X \cap A_0$. Consider the Hopf subalgebra $E$ of $A_0$ generated by $u_au_b^{-1}$. \vspace{-2pt}
As $\sqrt{p}\, e = \frac{1}{\sqrt{p}}\sum_{k} u_a^ku_b^{-k} \in X \cap E$, Corollary \ref{pisigmae}(i) implies $\frac{1}{\pi}(u_au_b^{-1}-e_0) \in X$. Hence
$$\frac{1}{\pi}(u_b^{-1}-e_0)=u_a^{-1}\Big(\frac{1}{\pi}(u_au_b^{-1}-e_0)-\frac{1}{\pi}(u_a-e_0)\Big) \in X.$$
By Proposition \ref{geomseries}, $\frac{1}{\sqrt{p}} \sum_{i} u_b^i \in X.$ Let $H_b$ be the Hopf subalgebra of $H$ generated by $u_b$ and $v_b$.
Arguing as we did for $H_a$ in the proof of Lemma \ref{uava}, we obtain that $\frac{1}{\sqrt{p}} \sum_{i} v_b^i \in X$. Applying Lemma \ref{mapT}, we have
$$T\Big(\frac{1}{\sqrt{p}} \sum_{i} v_b^i\Big)=ge_0 \in X$$
and we are done.
\epf

\subsection{The necessary condition}
We next derive that all basis elements of $H$ must be in the Hopf order $X$.
This will be key to establish the necessary condition of our main result and to prove later that a Hopf order of $H$, if exists, must be unique.

\begin{proposition}\label{basis}
All elements of the basis $\Bas$ in \eqref{basisH} of $H$ belong to $X$.
\end{proposition}

\pf From Lemmas \ref{e0e1} and \ref{uava}, we know that $e_0,e_1,u_a,v_a \in X$. We next see that $g \in X$. Take $\pi \in \Co$ such that $\pi^2=\zeta-1$ and set
$L=K(\pi), S=\Oint_L$. Then $X \otimes_R S$ is a Hopf order of $H_L:=H \otimes_K L$. Lemma \ref{ge1B} combined with Lemma \ref{ge0} yields that $g \in X \otimes_R S$.
We can identify $H \otimes_R L$ with $H_L$ via multiplication. Inside $H \otimes_R L$ we have
$(X+Rg)\otimes_R S \subseteq X \otimes_R S + Rg \otimes_R S=X \otimes_R S \subseteq (X+Rg)\otimes_R S$.
This equality holds indeed in $H \otimes_R S \subset H \otimes_R L$. Since $S$ is faithfully flat as an $R$-module, we obtain $X=X+Rg$. Therefore $g \in X$. \par \smallskip

It remains to prove that $u_b,v_b \in X$. We have that $u_bg=gu_a \in X$. Then $u_b=(u_bg)g \in X$ and consequently $\Delta(u_b) \in X \otimes_R X$.
If follows from the latter that $v_b \in X$ arguing for $H_b$ as we did for $H_a$ in the proof of Lemma \ref{uava}.
\epf

As a consequence of Lemma \ref{uava}, we get
$$\frac{1}{\sqrt{p}}\sum_{i} u_a^i+v_a^i \in X.$$
Let $E$ be the Hopf subalgebra of $H$ generated by the group-like element $h:=u_a+v_a$. Clearly, $E \simeq KC_p$ as Hopf algebras. Put $Z=E\cap X$ and
denote by $\Lambda$ the set of left integrals in the Hopf order $Z$ of $E$.

\begin{lemma}\label{intZ}
We have $\Lambda=R\big(\frac{1}{\sqrt{p}} \sum_{i} h^i\big).$
\end{lemma}

\pf
Obviously, $R(\frac{1}{\sqrt{p}}\sum_i h^i) \subseteq \Lambda$. \vspace{-1pt} For the reverse inclusion, let $\int \in \Lambda$.
There is $\lambda \in K$ such that $\int=\frac{\lambda}{\sqrt{p}}\sum_i h^i$. We will prove that $\lambda \in R$. Using Proposition \ref{basis},
$\varpi:=(\int\otimes \int)\Delta(g) \in X \otimes_R X$. Then $(\Gamma_2 \otimes_R \Gamma_2)(\varpi)\in R,$ with $\Gamma_2$ being the element defined in \eqref{elemW2}.
We next show that $(\Gamma_2 \otimes_R \Gamma_2)(\varpi)=\lambda^2$. \par \smallskip

Taking into account that $\Gamma_2$ vanishes on $A_0 \oplus A_1 \oplus gA_0$, it suffices to compute the part of $\varpi$ in $gA_1 \otimes gA_1$. We have:
$$\begin{array}{ll}
{\displaystyle (\Gamma_2 \otimes_R \Gamma_2)(\varpi)} &
 \stackrel{\eqref{elemW2}}{=} \ {\displaystyle \frac{\lambda^2}{p^3} \sum_{i,j,k,l} \sum_{r,s} (\gamma_{rr}B)(gv_a^{i+k}v_b^l)(\gamma_{ss}B)(gv_a^{j-l}v_b^k)} \vspace{3pt} \\
& \stackrel{\eqref{gammaijB}}{=} \ {\displaystyle \frac{\lambda^2}{p^2} \sum_{i,j,k,l} \sum_{r,s} \zeta^{r(i+k)}\delta_{l,0}\zeta^{s(j-l)}\delta_{k,0}} \vspace{3pt} \\
& \hspace{7pt} = \ \lambda^2.
\end{array}$$
So $\lambda^2\in R$ and thus $\lambda \in R$.
\epf

We can now establish the necessary condition in our main result from the previous lemma and Corollary \ref{pisigmae}(i):

\begin{proposition}\label{necessary}
Suppose that $H$ admits a Hopf order over $R$. Then there is an ideal $I$ of $R$ such that $I^{2(p-1)} = (p)$.
\end{proposition}

\subsection{The Hopf order}
Assume that there is an ideal $I$ of $R$ such that $I^{2(p-1)} = (p)$.
In this subsection we will construct from $I$ a Hopf order of $H$ which will turn out to be the only Hopf order.
Consider the fractional ideal $J:=I^{-1}=\{\alpha \in K : \alpha I \subseteq R\}$. By the unique factorization property in $R$,
from $I^{2(p-1)}=(p)=(\zeta-1)^{p-1}=(\sqrt{p})^2$, it follows that $I^2=(\zeta-1)$ and $I^{p-1}=(\sqrt{p})$.
Then $J^2=(\frac{1}{\zeta-1})$ and $J^{p-1}=(\frac{1}{\sqrt{p}})$. \par \smallskip

We need the following version of Proposition \ref{geomseries}:

\begin{proposition}\label{geomseries2}
Let $Z$ be an $R$-algebra and $z,e \in Z$. Assume that $ze=ez=z$. If $J(z-e) \subset Z$, then
$$\frac{1}{\sqrt{p}}\sum_i z^i \in Z.$$
\end{proposition}

\pf Put $\tilde{z}=z-e$, proceed like in the other proof and use that $\frac{1}{\sqrt{p}}\tilde{z}^{p-1} \in Z$.
\epf

\begin{theorem}\label{Xorder}
The $R$-subalgebra $Y$ of $H$ generated by $e_0,e_1,g,J(u_a-e_0),$ \linebreak $J(u_b-e_0), J(v_a-e_1),$ and $J(v_b-e_1)$ is a Hopf order of $H$.
\end{theorem}

\pf
We will first prove that $Y$ is finitely generated as an $R$-module. Observe that $J$ is finitely generated. Write
$$x_a=u_a-e_0, \quad  x_b=u_b-e_0, \quad y_a=v_a-e_1, \quad  y_b=v_b-e_1.$$
We have that $x_a,x_b,y_a,y_b \in Y$ because $IJ=R$. Since $e_0,e_1 \in Y,$ we also have $u_a,u_b,v_a,v_b \in Y$.
We next check that $(Jx_a)^n \subset \sum_{i=1}^{p-1} J^ix_a^i$ for $n \geq p$. The element $x_a$ satisfy $\sum_{i=1}^{p} {p \choose i}x_a^i=0$.
As $J^pJ^{p-2}=J^{2(p-1)}=(\frac{1}{p})$, we get $R=(J^pp)J^{p-2}$. Then $J^pp=I^{p-2} \subset R$. Hence
$$(Jx_a)^p=J^px_a^p \subset \sum_{i=1}^{p-1} J^p {p \choose i} x_a^i \subset \sum_{i=1}^{p-1} Rx_a^i \subset \sum_{i=1}^{p-1} J^ix_a^i.$$
The same holds for $x_b,y_a,$ and $y_b$. Consider now the equality:
$$\begin{array}{ll}
y_a y_b & = v_av_b - v_a - v_b + e_1 \vspace{2pt} \\
        & = \zeta v_bv_a - v_a - v_b + e_1 \vspace{2pt}  \\
        & = \zeta y_by_a + (\zeta-1)(y_a+ y_b+e_1).
\end{array}$$
Then, for $\alpha_a,\alpha_b \in J$ the coefficient of $e_1$ in $(\alpha_ay_a)(\alpha_by_b)$ belongs to $R$ because $J^2=(\frac{1}{\zeta-1})$.
Using the previous equality one can prove that any product of the form $(\beta_ay_a^k)(\beta_by_b^l)$ with $\beta_a \in J^k, \beta_b \in J^l$
can\vspace{-2pt} be expressed as an $R$-linear combination of elements in
$(J^iy_b^i)(J^jy_a^j)$ with $0 \leq i \leq l,\, 0 \leq j \leq k$.
Notice that the coefficient of $e_1$ always belong to $R$.
All these facts, together with the relations among $x_a,x_b,y_a,y_b,$ and $g$ inside $H$, show that $Y$ is finitely generated as an $R$-module.
More precisely, using that $J$ is finitely generated, the following elements generate $Y$ over $R$:
$$\begin{array}{l}
e_0,\, e_1,\, ge_0,\, ge_1,\, J^{i+j}(x_b^ix_a^j),\, J^{i+j}(gx_b^ix_a^j),\, J^{i+j}(y_b^iy_a^j),\, J^{i+j}(gy_b^iy_a^j), \vspace{3pt} \\
\hspace{10.8cm} i,j=0,\ldots,p-1. \end{array}$$
Removing the powers of $J$ from these elements, we obtain a $K$-basis of $H$ (we understand that $i,j$ are not simultaneously zero).
Hence $Y$ is an order of $H$. \par \medskip

We next prove that $Y$ is closed under comultiplication and antipode.
It is easy to check that the comultiplication of the $e$'s, $x$'s and $y$'s lie in $Y \otimes_R Y$, the counits of them lies in $R$, and $S(Y) \subset Y$.
For instance, for $\alpha \in J$ we have:
$$\begin{array}{l}
\Delta(\alpha x_a) = \alpha x_a \otimes u_a + \alpha y_a \otimes v_a + e_0 \otimes \alpha x_a + e_1 \otimes \alpha y_a \in Y \otimes_R Y, \vspace{5pt} \\
\Delta(\alpha x_b) = \alpha x_b \otimes u_b + e_0 \otimes \alpha x_b + \alpha y_b \otimes v_b^{p-1} \vspace{3pt} \\
 \hspace{1.8cm} +\sum\limits_{k=1}^{p-1} e_1 \otimes {p-1 \choose k} (\alpha y_b)y_b^{k-1} \in Y \otimes_R Y.
\end{array}$$
It only remains to show that $\Delta(g) \in Y \otimes_R Y$. For, we need to rewrite $\Delta(g)$ as an $R$-linear combination of elements in $Y \otimes_R Y$.
Recall from Equation \ref{comultg1} that $\Delta(g)$ consists of four summands. We treat each of them separately: \par \smallskip

$\bullet$ \underline{Part in $A_0 \otimes A_0$.} Consider the sum
$$\frac{1}{p^2}\sum_{i,j,k,l} \zeta^{jk-il} u_a^iu_b^j\otimes u_a^k u_b^l=\Big(\frac{1}{p}\sum_{i,l} \zeta^{-il} u_a^i \otimes u_b^l\Big)\Big(\frac{1}{p}\sum_{j,k} \zeta^{jk}u_b^j\otimes u_a^k\Big).$$
We argue on the first factor, the second one being similar. Replace $u_a$ and $u_b$ by $x_a+e_0$ and $x_b+e_0$ respectively and expand. The coefficient of
$x_a^r \otimes x_b^s$ equals $\frac{\zeta^{-1}}{p}$ if $r=s=p-1$. Then
$$\frac{\zeta^{-1}}{p}x_a^{p-1} \otimes x_b^{p-1}=\frac{\zeta^{-1}}{\sqrt{p}}x_a^{p-1} \otimes \frac{1}{\sqrt{p}}x_b^{p-1}$$
belongs to $Y \otimes_R Y$ because $\frac{1}{\sqrt{p}} \in J^{p-1}$. For either $r$ or $s$ different from $p-1$ we use the following argument.
The coefficient of $x_a^r \otimes x_b^s$ will be the same as the coefficient of $y_a^r \otimes y_b^s$ in the sum
$$\frac{1}{p}\sum_{i,l} \zeta^{-il} v_a^i\otimes v_b^l.$$
This in turn will be the same as the coefficient of $y_a^ry_b^s$ in the sum
\begin{align}
{\displaystyle \frac{1}{p} \sum_{i,l} \zeta^{-il}v_a^iv_b^l} &  = {\displaystyle \frac{1}{p} \sum_{i,l} v_b^l v_a^i} \notag \vspace{3pt} \\
 & = {\displaystyle \frac{1}{p}\Big(\sum_{l} v_b^l\Big)\Big(\sum_{i} v_a^i}\Big) \notag \vspace{3pt} \\
 & = {\displaystyle \frac{1}{p}\frac{(y_b +e_1)^p - e_1}{y_b}\frac{(y_a + e_1)^p -e_1}{y_a}}. \label{geomser}
\end{align}
We are using here the convention in the proof of Proposition \ref{glo} for these fractional expressions.
The coefficient of $y_a^ry_b^s$ in this sum contains the binomial coefficient $\binom{p}{k}$ for $k=1,\ldots,p-1$. Therefore the first factor belong to $Y \otimes_R Y$. \par \medskip

$\bullet$ \underline{Part in $A_0 \otimes A_1$.} We have the summand
$$\frac{1}{p}\sum_{k,l} \zeta^{-(k+l)k}u_a^ku_b^l\otimes v_a^{k+l}v_b^{k+l}=
\Big(\frac{1}{\sqrt{p}}\sum_{k} u_a^k \otimes v_b^k v_a^k\Big)\Big(\frac{1}{\sqrt{p}}\sum_{l} u_b^l \otimes v_a^l v_b^l\Big).$$
We show that each of the sums belongs to $Y \otimes_R Y$. We do it only for the first one.
For the second one proceed similarly. The coefficient of $y_a^r \otimes y_b^sy_a^t$ in this sum will be the same as the coefficient of $x_a^r \otimes x_b^sx_a^t$ in the sum
\begin{equation}\label{deltaga0a1}
\frac{1}{\sqrt{p}}\sum_{k} u_a^k \otimes u_b^ku_a^k=\frac{1}{\sqrt{p}}\sum_{k} (u_a \otimes u_bu_a)^k.
\end{equation}
Observe that $u_a \otimes u_bu_a \in Y \otimes_R Y$ and
$$\begin{array}{ll}
{\displaystyle J(u_a \otimes u_bu_a-e_0 \otimes e_0)} & = Jx_a \otimes x_bx_a + Jx_a \otimes x_b + Jx_a \otimes x_a + Jx_a \otimes e_0 \vspace{3pt} \\
                  & \phantom{aa} + e_0 \otimes (Jx_b)x_a+e_0 \otimes Jx_b+e_0 \otimes Jx_a \in Y \otimes_R Y.
\end{array}$$
This together Proposition \ref{geomseries2} yields that the sum belongs to $Y \otimes_R Y$.
\par \medskip

$\bullet$ \underline{Part in $A_1 \otimes A_0$.} We argue as before with the summand
$$\frac{1}{p}\sum_{k,l} \zeta^{(k+l)k} v_a^{k+l}v_b^{-(k+l)} \otimes u_a^ku_b^l=\Big(\frac{1}{\sqrt{p}}\sum_{k} v_b^{(p-1)k}v_a^k \otimes u_a^k\Big)
\Big(\frac{1}{\sqrt{p}}\sum_{l} v_a^l v_b^{(p-1)l} \otimes u_b^l\Big).$$
but using the following variation: $v_b^{p-1}=\bar{y}_b+e_1$ with $\bar{y}_b=\sum_{j=1}^{p-1} \binom{p-1}{j}y_b^j$ and $J \subset J^j$ for $j=1,\ldots,p-1$.
\medskip

$\bullet$ \underline{Part in $A_1 \otimes A_1$.} Consider the summand
$$\frac{1}{p}\sum_{k,l} v_a^kv_b^l \otimes v_a^{-l}v_b^k=\frac{1}{p}\sum_{k,l} v_a^kv_b^l \otimes v_a^{(p-1)l}v_b^k.$$
Write it in $H \otimes H^{op}$ as
$$\Big(\frac{1}{\sqrt{p}}\sum_{k}  v_a^k \otimes v_b^k \Big)\Big(\frac{1}{\sqrt{p}} \sum_{l} v_b^l
\otimes v_a^{(p-1)l}\Big)$$
and proceed as before. This finishes the proof.
\epf

\begin{proposition}\label{unique}
The Hopf order $Y$ is unique.
\end{proposition}

\pf Let $\pi \in \Co$ be such that $\pi^2=\zeta-1$ and set $L=K(\pi)$.
We will first prove that $H_L$ admits a unique Hopf order over $S=\Oint_L$ and derive the uniqueness for $H$ arguing as we did in Proposition \ref{basis}.
Write $I=(\pi)$. Then $I^{2(p-1)}=(p)$. Let $J \subset L$ be the inverse of $I$, which is generated by $\frac{1}{\pi}$.
We have seen in the precedent proof that the order $Y$ (over $S$) is generated as an algebra by $e_0,e_1,g$ and the elements
$$\tilde{x}_a:=\op (u_a-e_0), \quad \tilde{x}_b:=\op(u_b-e_0), \quad \tilde{y}_a:=\op(v_a-e_1), \quad \tilde{y}_b:=\op(v_b-e_1).$$

Let $X$ be any Hopf order of $H_L$. By Lemma \ref{uava} and Corollary \ref{pisigmae}(ii), $X$ must contain the element $\frac{1}{\pi}(u_a+v_a-1)$.
By Proposition \ref{basis}, $X$ contains all basis elements of $H_L$. Using multiplication by $e_0$ and $e_1$,
conjugation by $g$ and translation by the character $\rho:H_b \rightarrow K, u_b \mapsto 0, v_b \mapsto 1$,
we see that $X$ must contain $\tilde{x}_a,\tilde{y}_a,\tilde{x}_b,$ and $\tilde{y}_b$. Then $Y \subseteq X$ and thus $Y$ is a minimal Hopf order. \par \medskip

We know that $H_L$ is self-dual. Then $H_L^*$ has also a minimal order, which we denote by $Z$. This implies that $Z^{\star}$ is a maximal Hopf order of $H_L$.
Thus any Hopf order of $H_L$ lies between $Y$ and $Z^{\star}$. {We will prove that} $Y=Z^{\star}$. The $R$-submodule $\Lambda_Y$ of left integrals in $Y$ is spanned by
$\frac{1}{p}(1+g)\sum_{i,j} u_a^iu_b^j$. Then $\varepsilon(\Lambda_Y)=(2p)$. \vspace{-1pt} Using self-duality of $H_L$, we also have $\varepsilon(\Lambda_Z)=(2p)$.
Since $(\dim H)=(4p^2)$, by Proposition \ref{Larson1}, $\varepsilon(\Lambda_{Z^{\star}})=(2p)$. Proposition \ref{Larson2} yields $Y=Z^{\star}$. \par \medskip

Finally, let $X,X'$ be two Hopf orders of $H$. The Hopf orders $X \otimes_R S$ and $X' \otimes_R S$ of $H_L$ must be equal. Then $X \otimes_R S = (X+X') \otimes_R S = X' \otimes_R S.$
As $S$ is faithfully flat as an $R$-module, we obtain $X=X+X'=X'$ and we are done.
\epf
\vspace{0mm}

\begin{remark}
The precedent result shows that the behavior of orders for semisimple Hopf algebras can be quite different to that of group algebras.
When we take larger number fields, the number of Hopf orders of the group algebra on $C_p$ tends to infinity whereas the number of orders of $H$ is constantly $1$.
\end{remark}
\vspace{-4mm}

\begin{remark}
In \cite[Theorem 1.8]{Me} the second author proved that every semisimple Hopf algebra over a number field only admits finitely many Hopf orders over its ring of integers.
\end{remark}

\subsection{Main result}
We are finally in a position to prove our main result:

\begin{theorem}\label{main}
Let $p$ be an odd prime number and $K$ a number field containing a primitive $p$-th root of unity. Nikshych's Hopf algebra admits a Hopf order over $\Oint_K$,
which must be unique, if and only if there is an ideal $I$ of $\Oint_K$ such that $I^{2(p-1)} = (p)$. In particular, $K$ can not be neither a cyclotomic field nor an abelian extension of $\Q$ if a Hopf orders exist.
\end{theorem}

\pf The necessary condition was established in Proposition \ref{necessary}. The sufficient condition and uniqueness were proved in Theorem \ref{Xorder} and Proposition \ref{unique}. \par \smallskip

We prove that $K$ can not be a cyclotomic field if $H$ admits a Hopf order over $\Oint_K$. Let $I \subset \Oint_K$ be the given ideal such that $I^{2(p-1)} = (p)$. Suppose that $K$ is a cyclotomic field, say $K=\Q(\eta)$ with $\eta$ a primitive $m$-th root of unity. Since $(p)$ ramifies in $\Oint_K$, by \cite[Proposition 2.3]{W}, $p$ is a prime factor of $m$. Call $n$ the exponent with which $p$ occurs. By \cite[Theorem 4.40]{Nw},
there is a prime ideal $\mathfrak{P}$ of $\Oint_K$ appearing in the factorization of $(p)$ with exponent $e:=(p-1)p^{n-1}$.
The exponent of $\mathfrak{P}$ in the factorization of $I^{2(p-1)}$ will be $2l(p-1)$ for some $l \in \Na$. Then $p$ should be divisible by $2$, a contradiction. \par \smallskip

That $K$ can not be an abelian extension of $\Q$ in this case follows from the Kronecker-Weber Theorem.
\epf

\section{On orders of forms}
\setcounter{equation}{0}

Let $L/K$ be a Galois extension of fields with Galois group $\Gamma$. We have seen before that it could happen that Nikshych's Hopf algebra $H$ over $K$ does not admit an order over any cyclotomic ring of integers, but could a $L/K$-form of $H$ do? Namely, could there be another Hopf algebra $H'$ over $K$ such that $H' \otimes_K L \simeq H \otimes_K L$ and $H'$ admits an order over some cyclotomic ring of integers? We will show in this last section that the answer to this question is affirmative. \par \smallskip

We first recall from \cite[Proposition 1.1]{CDL} and \cite[Proposition 1]{RTW} some basics about Galois descent in the Hopf algebra setting. Put $H_L=H \otimes_K L$. Given $\gamma \in \Gamma$, a {\it Hopf $\gamma$-automorphism} of $H_L$ is a $K$-linear automorphism $f:H_L \rightarrow H_L$ which satisfies:
\begin{enumerate}
\item[(1)] $f$ is $\gamma$-semilinear, i.e., $f(\alpha h)=\gamma(\alpha)f(h)$ for all $\alpha \in L, h \in H_L$.
\item[(2)] $f$ is compatible with the multiplication, comultiplication, and antipode.
\item[(3)] $f(1_{H_L})=1_{H_L}$.
\item[(4)] $\varepsilon f = \gamma \varepsilon.$
\end{enumerate}

According to Galois descent, $L/K$-forms of $H$ correspond to group homomorphisms $\Phi:\Gamma \rightarrow Aut_K(H_L), \gamma \mapsto \Phi_{\gamma}$ such that $\Phi_{\gamma}$ is a Hopf $\gamma$-automorphism for all $\gamma \in \Gamma$.
For such a $\Phi$ the set of invariants $(H_L)^{\Gamma}$ is a Hopf algebra over $K$ and the natural map $(H_L)^{\Gamma} \otimes_K L \rightarrow H_L$ is an isomorphism of Hopf algebras.\par \smallskip

Our goal is to prove the following:

\begin{theorem}\label{descents}
Let $\zeta_n \in \Co$ be a primitive $n$-th root of unity, with $n$ divisible by $p$. Consider Nikshych's Hopf algebra $H$ as defined over $\Q(\zeta_n)$. Let $w \in \Z[\zeta_n]$ and $t \in \Co$ be such that $w$ is invertible and $t^2=w({\zeta_p}-1)$. Assume that\vspace{-1pt} there is $d \in \Z[\zeta_n]$ such that $\frac{1}{2}(d+t)\in \Ow_{\Q(\zeta_n,t)}$. Then, $H$ admits a $\Q(\zeta_n,t)/\Q(\zeta_n)$-form $H'$ which in turn admits an order over $\Z[\zeta_n]$.
\end{theorem}

\begin{proof}
Set $L=\Q(\zeta_n,t)$. We will construct $H'$ and show that the unique order $Y$ of $H_L$ descents to an order of $H'$ over $\Z[\zeta_n]$. The Galois group $\Gamma$ of $L/\Q(\zeta_n)$ is isomorphic to $\Z/2\Z$. We denote the generator by $\gamma$. Bear in mind the Hopf automorphism $\sigma$ of $H$ of order two given by
$$\sigma(u_{\theta})=u_{\theta}^{-1}, \, \sigma(v_{\theta})=v_{\theta}^{-1} \textrm{ for } {\theta}=a,b, \textrm{ and } \sigma(g)=g.$$
We can define a Hopf $\gamma$-automorphism $\sigma'$ of $H_L$ by $\sigma'(h \otimes \alpha)=\sigma(h) \otimes \gamma(\alpha)$ for all $h \in H, \alpha \in L$. Let $\Phi:\Gamma\rightarrow Aut_K(H_L)$ be the group morphism mapping $\gamma$ to $\sigma'$. Consider the form $H'$ of $H$ given by $H'=(H_L)^{\Gamma}$. \par \smallskip

We claim that the order $Y$ of $H_L$ descents to an order $Y':=Y^{\Gamma}$ of $H'$ over $\Z[\zeta_n]$.
It is enough to check that the natural map $\rho:Y^{\Gamma} \ot_{\Z[\zeta_n]} \Ow_L \rightarrow Y$ is an isomorphism (this will ensure us that $Y^{\Gamma}$ is really a Hopf order). Since $\rho$ is injective, it suffices to check the surjectivity. We have seen in Proposition \ref{unique} that $Y$ is generated over $\Ow_L$ by $e_0,e_1,g,$ and
$$\tilde{x}_a:=\frac{1}{t}(u_a-e_0),\quad \tilde{x}_b:=\frac{1}{t}(u_b-e_0), \quad \tilde{y}_a:=\frac{1}{t}(v_a-e_1), \quad \tilde{y}_b:=\frac{1}{t}(v_b-e_1).$$
Clearly, $e_0,e_1,g \in {\Ima \rho}$ as they are invariants. We will show that ${\Ima \rho}$ contains the rest of the generators. Since ${\Ima \rho}$ is a subring of $Y$, this will give ${\Ima \rho}=Y$. Let us show that ${\tilde{x}_a} \in {\Ima \rho}$. The proof for the other generators is similar.
The element  $q:=\frac{1}{t^2}(2e_0-u_a-u_a^{-1})= -\tilde{x}_a^2u_a^{-1}$ belongs to $Y^{\Gamma}$. Since $\gamma(t)=-t$,
a direct calculation reveals that $\sigma'(\tilde{x}_a)= \tilde{x}_a+tq$. Set $z=\tilde{x}_a+\frac{1}{2}(d+t)q$.
One can easily check that $z\in Y^{\Gamma}$, and therefore $z\in {\Ima \rho}$. Finally, $\tilde{x}_a=z-\frac{1}{2}(d+t)q$, and $\frac{1}{2}(d+t)q \in {\Ima \rho}$, so $\tilde{x}_a\in {\Ima \rho}$ as well, as desired.
\end{proof}
\medskip

With the previous theorem in hand, we will describe an example in which an order of a form does exist.

\begin{example}
Consider the case $p=7$ and $n=28$. Let $\zeta:=\zeta_{28}$ be a primitive $28$-th root of unity. A computation done by Dror Speiser with the computer algebra system \href{http://magma.maths.usyd.edu.au/magma/}{MAGMA} showed that if ${w}$ is the inverse of the element
$$\begin{array}{l}
\phantom{+} 21747826028152\zeta^{11} - 25061812676688\zeta^{10} + 5371269408312\zeta^{9} -
    2754700868376\zeta^8 \vspace{5pt} \\
+ 21747826028152\zeta^7 - 22307111808312\zeta^6 + 4963799311635\zeta^4 + 12069132874072\zeta^3 \vspace{5pt} \\
- 11153555904156\zeta^2 - 12069132874072\zeta + 17343312496677
\end{array}$$
and $d=1$, then the condition of the theorem holds. {We take $t$ such that \linebreak $t^2=w(1-\zeta^4)$.} We thus have an order over $\Z[\zeta]$ of a form of $H_7$. \par \smallskip

Then $H_7$, as defined over the complex numbers, admits an order over a cyclotomic ring of integers.
\end{example}

The following questions remain open:

\begin{questions}
Does there exist a value of $p$ for which Nikshych's Hopf algebra $H_p$, as defined over the complex numbers,  does not admit an order over any cyclotomic ring of integers? More generally, does there exist a complex semisimple Hopf algebra which admits an order over a number ring but not over any cyclotomic ring of integers?
\end{questions}
\bigskip

\subsection*{Acknowledgements}
The first author is supported by grant MTM2014-54439-P from MICINN and FEDER and by the research group FQM0211 from Junta de Andaluc\'{\i}a. The second author was supported by the Danish National Research Foundation (DNRF) through the Centre for Symmetry and Deformation. The authors are grateful to Dror Speiser for doing the previous computer calculation and to Bjorn Poonen for a conversation about the number theoretical condition in Theorem \ref{descents}. The authors are finally indebted to the referee for his/her comments and suggestions, which helped to improve substantially the presentation of the results.

\end{document}